\documentclass{article}
\usepackage[latin1]{inputenc}
\usepackage[english]{babel}
\usepackage[dvips]{graphics}
\usepackage{epsfig}
\usepackage{fullpage}
\usepackage{amsfonts}
\usepackage{multicol}
\usepackage{latexsym,amssymb}     
\usepackage{amsmath}
\usepackage{indentfirst}
\usepackage{makeidx}
\usepackage{theorem}
\usepackage{bbold}
\usepackage{psfrag}
\usepackage{enumitem}

\newcommand{\rr}{\mathbb{R}}
\newcommand{\ee}{\mathbb{E}}

\newcommand{\nn}{\mathbb{N}}
\newcommand{\pp}{\mathbb{P}}

\newcommand{\ds}{\displaystyle}
\newcommand{\ind}{\mathbb{1}}

\def\BS{Black \& Scholes }

\newcommand{\blanc}[1]{\vspace{#1\baselineskip}}

\newtheorem{nt_theorem}{Theorem}
\newenvironment{theorem}{\blanc{1.5}\begin{nt_theorem}---}{\end{nt_theorem}\blanc{1.5}}

\newtheorem{nt_proposition}[nt_theorem]{Proposition}
\newenvironment{proposition}{\blanc{1.5}\begin{nt_proposition}---}{\end{nt_proposition}\blanc{1.5}}

\newtheorem{nt_corollaire}[nt_theorem]{Corollary}
\newenvironment{Corollary}{\blanc{1.5}\begin{nt_corollaire}---}{\end{nt_corollaire}\blanc{1.5}}

\newtheorem{nt_definition}[nt_theorem]{Definition}

\newtheorem{nt_lemma}[nt_theorem]{Lemma}
\newenvironment{lemma}{\blanc{0.5}\begin{nt_lemma}---}{\end{nt_lemma}\blanc{0.5}}

\newtheorem{nt_jeulemma}[nt_theorem]{Jeulin's Lemma}

\newtheorem{nt_conjecture}[nt_theorem]{Conjecture}

\newtheorem{nt_remark}[nt_theorem]{Remark}
\newenvironment{remark}{\blanc{0.2}\begin{nt_remark}---}{\end{nt_remark}\blanc{0.2}}

\newenvironment{proof}{{\textit{Proof : }}}{\hfill$\Box$
\\\bigskip}

\newenvironment{proofOFT}{{\textit{Proof of the theorem : }}}{\hfill$\Box$
\\\bigskip}

\makeatletter
\newcounter{hypo}
\newcommand*{\dohypo}{\textbf{(${\mathcal H}$\thehypo)}}
\newenvironment{hypo}[1][]{%
 
 \refstepcounter{hypo}
 \list{}{%
   \settowidth{\labelwidth}{\dohypo}%
   \setlength{\labelsep}{10pt}%
   \setlength{\leftmargin}{\labelwidth}
   \advance\leftmargin\labelsep%
 }%
\item[\dohypo  #1]%
}{%
 \endlist
}

\def\hypref#1{\hyperref[hyp:#1]{(${\mathcal H}$\ref*{hyp:#1})}}
\def\hypreff#1#2{\hyperref[hyp:#2]{(${\mathcal H}$\ref*{hyp:#1}-{\it
\ref*{hyp:#2})}}}

\makeatother

\begin{document}

\vspace{5mm}

\begin{center}
\textbf{\Large{High order discretization schemes for stochastic volatility models}}\\
\end{center}

$\,$

\begin{center}
\emph{Benjamin Jourdain} and \emph{Mohamed Sbai}\footnote{Université
Paris-Est, CERMICS, Projet MathFi ENPC-INRIA-UMLV. This research
benefited from the support of the

``Chaire Risques Financiers'', Fondation du Risque and the ANR
program BigMC.

Postal address : 6-8 av. Blaise Pascal, Cité Descartes,
Champs-sur-Marne, 77455 Marne-la-Vallée Cedex 2.

E-mails : \underline{jourdain@cermics.enpc.fr} and
\underline{sbai@cermics.enpc.fr} }
\end{center}

\vspace{5mm}

\begin{abstract}
In typical stochastic volatility models, the process driving the
volatility of the asset price evolves according to an autonomous
one-dimensional stochastic differential equation. We assume that the
coefficients of this equation are smooth. Using Itô's formula, we
get rid, in the asset price dynamics, of the stochastic integral
with respect to the Brownian motion driving this SDE. Taking
advantage of this structure, we propose
\begin{itemize}
\item[-] a scheme, based on the Milstein discretization of this SDE, which converges with order
one to the asset price dynamics for an appropriate notion of convergence that
we call weak trajectorial convergence,
\item[-] a scheme, based on the Ninomiya-Victoir discretization of this SDE,
with order two of weak convergence to the asset price.
\end{itemize}
We also propose a specific scheme with improved convergence properties when the
volatility of the asset price is driven by an Ornstein-Uhlenbeck process. We
confirm the theoretical rates of convergence by numerical experiments and show
that our schemes are well adapted to the multilevel Monte Carlo method
introduced by Giles (\emph{Multilevel Monte Carlo path simulation}. Operations
Research, 56:607-617, 2008).
\end{abstract}

\vspace{5mm}

\section*{Introduction}

There exists an extensive literature on numerical integration schemes for
stochastic differential equations. To start with, we mention, among many
others, the work of Talay and Tubaro \cite{TalayTubaro} who first established
an expansion of the weak error of the Euler scheme for polynomially growing
functions allowing for the use of Romberg extrapolation. Bally and Talay
\cite{BallyTalay1} extended this result to bounded measurable functions and
Guyon \cite{Guyon} extended it to tempered stable distributions. More recently,
many discretization schemes of higher weak convergence order have appeared in
the literature. Among others, we cite the work of Kusuoka
\cite{Kusuoka1,Kusuoka2}, the Ninomiya and Victoir \cite{NinomiyaVictoir}
scheme which we will use hereafter, the Ninomiya and Ninomiya
\cite{NinomiyaNinomiya} scheme
and the scheme based on cubature on Wiener spaces of Lyons and Victoir \cite{LyonsVictoir}.\\
Concerning strong approximation, the Milstein scheme has order one of strong
convergence. Unfortunately, it involves the simulation of iterated Brownian
integrals unless a restrictive commutativity condition is satisfied. Under
ellipticity, Cruzeiro \textit{et al.} \cite{CruzeiroMalliavinThalmaier} have
recently proposed a discretization scheme which gets rid of these iterated
integrals and has nice strong convergence properties. More precisely, for each
number of time steps, there exists a Brownian motion different from the one
giving the Brownian increments involved in the scheme such that the strong
error between the scheme and the stochastic differential equation driven by
this new Brownian motion is of order one. We call such a property weak
trajectorial convergence of order one. Weak trajectorial error estimation is
exactly what is needed to control the discretization bias for the computation
of path dependent option prices.

Stochastic volatility models, which have now become a standard in
the market, are an eloquent example of the use of stochastic
differential equations in finance. In our study, we will consider
the following specification which nests many current
stochastic volatility models :
\begin{equation}
\label{SV1}
\left\{\begin{array}{rcl} dS_t&=&rS_tdt+f(Y_t)S_t\left(\rho dW_t+\sqrt{1-\rho^2}dB_t\right) ;\quad S_0=s_0>0\\[1mm]
dY_t&=&b(Y_t)dt+\sigma(Y_t)dW_t;\quad Y_0=y_0,
\end{array}\right.
\end{equation}
where $(S_t)_{t\in[0,T]}$ is the asset price, $r$ the instantaneous interest
rate, $(B_t)_{t\in[0,T]}$ and $(W_t)_{t\in[0,T]}$ are independent standard
one-dimensional Brownian motions, $\rho\in[-1,1]$ is the correlation between
the Brownian motions respectively driving the asset price and the process
$(Y_t)_{t\in[0,T]}$  which solves a one-dimensional autonomous stochastic
differential equation. The volatility process is $(f(Y_t))_{t\in[0,T]}$ where
the transformation function $f$ is usually taken positive and strictly
monotonic in order to ensure that the effective correlation between the stock
price and the volatility keeps the same sign (the function $\sigma$ usually
takes nonnegative values). In the literature, the development of specific
discretization schemes for stochastic volatility models has only received
little attention. We mention nevertheless the work of Kahl and Jäckel
\cite{KahlJackel} who discussed various numerical integration methods and
proposed a simple scheme with order $1/2$ of strong convergence like the
standard Euler scheme but with a smaller multiplicative constant. Also the
numerical integration of the CIR process and of the Heston model received a
particular attention because of the inadequacy of the Euler scheme due to the
fact that both $f$ and $\sigma$ are equal to the square root function (see for
example Deelstra and Delbaen \cite{DeelstraDelbaen}, Alfonsi \cite{Alfonsi},
Kahl and Schurz \cite{KahlSchurz}, Andersen \cite{Andersen1}, Berkaoui
\textit{et al.} \cite{BossyDiop}, Ninomiya and Victoir \cite{NinomiyaVictoir},
Lord \textit{et al.} \cite{LordKoekoekVanDijk}, Alfonsi \cite{Alfonsi2}). An
exact simulation technique for the Heston model was also proposed by Broadie
and Kaya \cite{BroadieKaya}.

In the present paper, we assume that the functions $f$, $\sigma$ and
$b$ are smooth which means that we do not deal with the Heston model
where $f(y)=\sqrt{y}, \,b(y)=\kappa(\theta-y)$ and
$\sigma(y)=\nu\sqrt{y}$. As an example of stochastic volatility
models that fall within the scope of our study, let us mention

\begin{itemize}
\item Hull\&White \cite{HullWhite} model ($\rho=0$) and Wiggins
    \cite{Wiggins} ($\rho \neq 0$)
\[\left\{\begin{array}{rcl} dS_t&=&rS_tdt+\sqrt{Y_t}S_t\left(\rho dW_t+\sqrt{1-\rho^2}dB_t\right)\\[1mm]
dY_t&=&\mu Y_t dt+ \zeta Y_t dW_t
\end{array}\right.\]
which can be expressed as (\ref{SV1}) with $f(y)=\sqrt{y},
\,b(y)=\mu y$ and  $\sigma(y)=\zeta y$. Note that it can also be
seen as (\ref{SV1}) with $f(y)=e^y,
\,b(y)=\frac{\mu}{2}-\frac{\zeta^2}{4}$ and
$\sigma(y)=\frac{\zeta}{2}$.
\item Scott's model \cite{Scott} which generalizes the Hull\&White model
\begin{equation}
\label{Scott}\begin{array}{c}
\ds \left\{\begin{array}{rcl} dS_t&=&rS_tdt+\sigma_0e^{Y_t}S_t\left(\rho dW_t+\sqrt{1-\rho^2}dB_t\right)\\[1mm]
dY_t&=&\kappa(\theta-Y_t) dt+ \nu dW_t
\end{array}\right.\\[5mm]
\ds \Rightarrow f(y)=\sigma_0e^y, \,b(y)=\kappa(\theta-y) \text{ and
} \sigma(y)=\nu.
\end{array}
\end{equation}
\item Stein\&Stein model \cite{SteinStein}
\[\begin{array}{c}
\ds \left\{\begin{array}{rcl} dS_t&=&rS_tdt+{Y_t}S_t\left(\rho dW_t+\sqrt{1-\rho^2}dB_t\right)\\[1mm]
dY_t&=&\kappa(\theta-Y_t) dt+ \nu dW_t
\end{array}\right.\\[5mm]
\ds \Rightarrow f(y)=y, \,b(y)=\kappa(\theta-y) \text{ and }
\sigma(y)=\nu.\end{array}\]
\item Quadratic Gaussian model
\[\begin{array}{c}
\ds\left\{\begin{array}{rcl} dS_t&=&rS_tdt+Y_t^2S_t\left(\rho dW_t+\sqrt{1-\rho^2}dB_t\right)\\[1mm]
dY_t&=&\kappa(\theta-Y_t) dt+ \nu dW_t
\end{array}\right.\\[5mm]
\ds \Rightarrow f(y)=y^2, \,b(y)=\kappa(\theta-y) \text{ and }
\sigma(y)=\nu.\end{array}\]
\end{itemize}

Our aim is to take advantage of the structure of \eqref{SV1} to
construct and analyse simple and robust ad hoc discretization
schemes which have nice convergence properties. For a start, we make
a logarithmic change of variables for the asset : the
two-dimensional process
$(X_t:=\log\left(S_t\right),Y_t)_{t\in[0,T]}$ solves the following
SDE
\begin{equation}
\label{SV2} \left\{\begin{array}{rcl} \ds dX_t&=&\ds
(r-\frac{1}{2}f^2(Y_t))dt+f(Y_t)\left(\rho
dW_t+\sqrt{1-\rho^2}dB_t\right);\quad X_0=\log(s_0).\\[1mm]
\ds dY_t&=&\ds b(Y_t)dt+\sigma(Y_t) dW_t;\quad Y_0=y_0.
\end{array}\right.\end{equation}

Our main idea is to get rid in the first equality of the stochastic
integral involving the common Brownian motion $(W_t)_{t\in[0,T]}$.
In all what follows, we assume that

$\,$\\
$\left(\mathcal{A}\right) \quad f$ and $\sigma$ are $\mathcal{C}^1$
functions and $\sigma
>0$.

$\,$\\
One can then define the primitive $F(y)=\int_0^y \frac{f}{\sigma}(z)
dz$ and apply Itô's formula to get
\[dF(Y_t)=\frac{f}{\sigma}(Y_t)dY_t+\frac{1}{2}(\sigma f'-f\sigma')(Y_t)dt.\]
Therefore $(X_t,Y_t)_{t\in[0,T]}$ solves
\begin{equation}
\label{SV3} \left\{\begin{array}{rcl} dX_t&=&\rho dF(Y_t)+h(Y_t)
dt+\sqrt{1-\rho^2}f(Y_t)dB_t\\[1mm]
dY_t&=&b(Y_t)dt+\sigma(Y_t)dW_t
\end{array}\right.,
\end{equation}
where $h:y\mapsto
r-\frac{1}{2}f^2(y)-\rho(\frac{b}{\sigma}f+\frac{1}{2}(\sigma
f'-f\sigma'))(y)$. We discretize the autonomous SDE satisfied by $Y$
using a scheme with high order of strong or weak convergence
depending on whether one is interested in path-dependent or vanilla
options. Then, in the dynamics of $X$, we only need to discretize
the standard integral $\int_0^T h(Y_s)ds$ and the stochastic
integral $\int_0^T f(Y_t)dB_t$ where $(Y_t)_{t\in[0,T]}$ and
$(B_t)_{t\in[0,T]}$ are independent.

We recall that weak convergence is the right notion to analyse
the discretization bias for plain vanilla options whereas weak
trajectorial convergence permits to deal with path-dependent
options. The first section of the paper is devoted to path-dependent
options. Combining the Milstein discretization of the
one-dimensional SDE satisfied by $(Y_t)_{t\in[0,T]}$ with an
appropriate discretization of the integral $\int_0^T f(Y_t)dB_t$
based on the independence of $(Y_t)_{t\in[0,T]}$ and
$(B_t)_{t\in[0,T]}$, we obtain a scheme with order one of weak
trajectorial convergence under several assumptions, the most
restrictive one being that $f^2$ is bounded away from 0. When
$(Y_t)_{t\in[0,T]}$ follows an Ornstein-Uhlenbeck process, which is
the case for all the models cited above, the order one is preserved
when replacing the Milstein discretization with exact simulation.
Unfortunately, the assumption $f^2$ bounded away from 0 is not
satisfied by the stochastic volatility models cited above. For Scott
and Hull \& White models, where $f$ is positive, we manage to prove
that the order one is preserved. For the quadratic Gaussian model,
taking advantage of the flatness of $f(y)=y^2$ around the origin where this function vanishes,
we are able to prove that the order of convergence is $1-\varepsilon$ for any $\varepsilon>0$.
Finally, in the Stein \& Stein model, the fact that the derivative of $f(y)=y$ does not vanish where this function is zero weakens
the  order of convergence : the order $\frac{3}{4}-\varepsilon$ in $L^2$ obtained by our theoretical analysis is confirmed by our numerical experiments.

In the second section, using the Ninomiya-Victoir discretization of the SDE
satisfied by $(Y_t)_{t\in[0,T]}$, we construct a scheme with order two of weak
convergence. Since the SDE satisfied by $Y$ is one-dimensional, the
Ninomiya-Victoir scheme only involves two one-dimensional ODEs whose solutions
are available in closed form. The last section is devoted to numerical
experiments which confirm the theoretical rates of convergence. We also compare
the time needed by the different schemes to achieve a given precision in the
multilevel Monte Carlo computation of a plain vanilla Call option and a
lookback option. The multilevel Monte Carlo method proposed recently by Giles
\cite{Giles1} automatically balances the bias and the statistical error and
optimally takes advantage of both the weak and the strong convergence
properties of the schemes to accelerate the computation. Somehow surprisingly,
the strong convergence order has a dominating effect on its efficiency. We are
able to exhibit an explicit coupling with order one of convergence between our
weak trajectorial schemes with $N$ and $2N$ steps (see Remark 5). With this
coupling, the multilevel Monte Carlo estimator behaves as if the scheme had
order one of strong convergence and the computation time needed to achieve the
root-mean-square error $\epsilon>0$ is ${\cal O}(\epsilon^{-2})$ (see
\cite{Giles1}). For high levels of precision our schemes turn out to be more
efficient than the Euler and the Kahl-Jäckel schemes for both the vanilla Call
and the lookback option. The reason is that their better convergence properties
compensate for the increase of computation effort at each step.

\subsubsection*{Notations}We will consider, for a number of
time steps $N \geq 1$, the uniform subdivision
$\prod_N=\{0=t_0<t_1<\dots<t_N=T\}$ of $[0,T]$ with the
discretization step $\delta_N=\frac{T}{N}$.\\
We denote by $\underline{\psi}$ the greatest lower bound of the
function $\psi:y\mapsto f^2(y)$ and by $\overline{\psi}$ its lowest
upper bound. We also introduce the following notation :
\[\widehat{\psi}(y)=\left\{\begin{array}{ll}\frac{3}{2}
f^2(y)&\text{ if } \overline{\psi}=\infty\\[3mm]
\overline{\psi}&\text{ otherwise.}\end{array}\right.\]

\section{An efficient scheme for path dependent options pricing}
Building a first order strong convergence scheme for a two dimensional SDE is
not an obvious task. Even the ad hoc schemes provided by Kahl and Jäckel
\cite{KahlJackel} exhibit a strong convergence of order $\frac{1}{2}$.

Actually, the natural candidate for this purpose is the Milstein
scheme. Unfortunately, the commutativity condition which permits to
implement it amounts to $\sigma f'=0$ in our setting. This condition
is typically true when either $f$ is constant or $\sigma=0$. Both
cases are of no practical interest since they lead to a
deterministic volatility.

However, since the inherent Brownian motion is not essential for applications
in finance, the usual strong convergence criterion is not crucial for
estimating the error of a scheme in pricing a path dependent option. What is
more relevant is the approximation in law of the whole trajectory of the
process considered for instance by Cruzeiro \textit{et al.}
\cite{CruzeiroMalliavinThalmaier}. Using an ingenious rotation of the Brownian
motion, these authors have constructed a discretization scheme allowing for a
weak convergence on the whole trajectory of order one which avoids the
simulation of the iterated stochastic integrals.

For the SDE (\ref{SV2}), the discretization scheme of Cruzeiro,
Malliavin and Thalmaier writes as
\begin{equation}
\label{CMT}
\begin{array}{l}
X^{CMT}_{t_{k+1}}=X^{CMT}_{t_k}+\left(r-\frac{f^2(Y^{CMT}_{t_k})}{2}\right)\delta_N+\rho
f(Y^{CMT}_{t_k})\Delta W_{k+1}+\frac{\rho}{2}\sigma
f'(Y^{CMT}_{t_k})\Delta
W_{k+1}^2\\[3mm]
\quad \quad +\sqrt{1-\rho^2}\sigma f'(Y^{CMT}_{t_k})\Delta
W_{k+1}\Delta B_{k+1}+\sqrt{1-\rho^2}f(Y^{CMT}_{t_k})\Delta
B_{k+1}-\frac{\rho}{2}\sigma f'(Y^{CMT}_{t_k})\Delta
B_{k+1}^2\\[5mm]
Y^{CMT}_{t_{k+1}}=Y^{CMT}_{t_k}+\left(b(Y^{CMT}_{t_k})+\frac{1}{2}(\frac{\sigma^2f'}{f}-\sigma\sigma')(Y^{CMT}_{t_k})\right)
\delta_N+\sigma(Y^{CMT}_{t_k})\Delta W_{k+1}\\[3mm]
\quad \quad+\frac{1}{2}\sigma\sigma'(Y^{CMT}_{t_k})\Delta
W_{k+1}^2-\frac{\sigma^2f'}{2f}\Delta B_{k+1}^2
\end{array}
\end{equation}
where $\Delta W_{{k+1}}=W_{t_{k+1}}-W_{t_k}$ and $\Delta
B_{k+1}=B_{t_{k+1}}-B_{t_k}$ correspond to the Brownian increments.

We set out to construct a much simpler scheme having the same order
of weak trajectorial convergence by taking advantage of the
particular structure of the SDE defining stochastic volatility
models. We first begin with the general case of any process
$(Y_t)_{t\in[0,T]}$ driving the volatility and then consider the
case of an Ornstein-Uhlenbeck process where we obtain more precise
results.

\vspace{2mm}

\subsection{General case}


A discretization scheme will naturally involve the Brownian
increments. Thanks to the independence between $(Y_t)_{t\in[0,T]}$
and $(B_t)_{t\in[0,T]}$, we can construct a vector
$(\widetilde{X}_{t_0},\dots,\widetilde{X}_{t_N})$ using only
$(\Delta B_{1},\dots,\Delta B_N)$ and $(Y_t)_{t\in[0,T]}$, which has
exactly the same law as $(X_{t_0},\dots,X_{t_N})$ :

\begin{lemma}
$\ds \forall 0\leq l < N,$ let $v_l=\frac{1}{\delta_{N}}
\int_{t_{l}}^{t_{l+1}} \!\!\psi(Y_{s})ds$.  The vector
$(\widetilde{X}_{t_0},\dots,\widetilde{X}_{t_N})$ defined by
 \[\begin{array}{l}
   \ds \widetilde{X}_{t_0}=X_{t_{0}} \\
\ds \forall 1\leq k\leq N,
\widetilde{X}_{t_{k}}=\widetilde{X}_{t_0}+\rho(F(Y_{t_{k}})-F(Y_{t_{0}}))
 +\int_{{0}}^{t_{k}} \!\!h(Y_{s})ds+\sqrt{1-\rho^{2}}\sum_{l=0}^{k-1} \sqrt{v_l} \,\Delta
 B_{l+1}
 \end{array}\]
has the same law as $(X_{t_0},\dots,X_{t_N})$.
\end{lemma}
\begin{proof}
 The proof is elementary. Conditionally on $Y$, the two vectors are Gaussian vectors
 with the same mean and covariance matrix.
\end{proof}

In order to approximate $(\widetilde{X}_{t_k})_{0\leq k\leq N}$, one
needs to discretize $v_k$ for $k \in \{0,\dots,N-1\}$. If
$(v_k^N)_{0\leq k\leq N-1}$ is an approximation of $(v_k)_{0\leq
k\leq N-1}$, then by Doob's inequality
\[\begin{array}{rcl}
\ds \ee\left[\sup_{0\leq k\leq N-1} \left(\sum_{l=0}^k
\left(\sqrt{v_l}-\sqrt{v_l^N}\right)\, \Delta
B_{l+1}\right)^2\right] &\leq&\ds 4 \delta_N \sum_{k=0}^{N-1}
\ee\left[\left(\sqrt{v_l}-\sqrt{v_l^N}\right)^2\right]\\[2mm]
&\leq&\ds \frac{1}{\underline{\psi}} \,\delta_N
\sum_{k=0}^{N-1} \ee\left[\left(v_l-v_l^N\right)^2\right]\\
\end{array}\]
as soon as $\underline{\psi}=\inf_x \psi(x)$ is assumed to be
positive and, $\forall 0\leq k\leq N-1, v_k^N$ is greater than
$\underline{\psi}$. Consequently, to obtain a scheme with order one
of strong convergence for $(\widetilde{X}_{t_k})_{0\leq k\leq N}$,
one needs that $\forall 0\leq k\leq N-1,
\ee\left[\left(v_k-v_k^N\right)^2\right] =
\mathcal{O}\left(\frac{1}{N^2}\right)$. According to the treatment
of the term $\overline{I}^j_2$ defined by (\ref{termIj2}) in the
proof of the Theorem \ref{THfort} below, one has $\forall 0\leq
k\leq N-1,$
\begin{equation}
\ee\left[\left(v_k-\left(\psi(Y_{t_k})+\frac{\sigma
\psi'(Y_{t_k})}{\delta_N}\int_{t_k}^{t_{k+1}}
(W_s-W_{t_k})ds\right)\right)^2\right]=\mathcal{O}\left(\frac{1}{N^2}\right).
\end{equation}
This equality still holds true when replacing $Y$ by a scheme with
order one of strong convergence in the term with sign minus of the
left hand side. Better still,
$\left(F(Y_{t_k})+\int_0^{t_k}h(Y_s)ds\right)_{0\leq k\leq N}$ is
approximated with strong order one when replacing $Y$ by such a
scheme and using a rectangular discretization for the integral in
time.

\vspace{2mm}

For all these reasons, we choose the Milstein scheme for $Y$ :
\[
\forall 0\leq k \leq N-1, \,
\widetilde{Y}^N_{t_{k+1}}=\widetilde{Y}^N_{t_k}+b(\widetilde{Y}^N_{t_k})\delta_N+\sigma(\widetilde{Y}^N_{t_k})
\Delta W_{k+1}+\frac{1}{2}\sigma
\sigma'(\widetilde{Y}^N_{t_k})\left(\Delta
W_{{k+1}}^2-\delta_N\right); \quad \widetilde{Y}^N_{t_0}=y_0.
\]
and we write our scheme as follows

\vspace{2mm}

\framebox{\textbf{WeakTraj\_1 scheme}}
\begin{equation}
\label{strong} \begin{array}{l} \ds
\widetilde{X}_{t_{k+1}}^N=\widetilde{X}_{t_{k}}^N+\rho
\left(F(\widetilde{Y}^N_{t_{k+1}})-F(\widetilde{Y}^N_{t_{k}})\right)+ \delta_N h(\widetilde{Y}^N_{t_k})\\[4mm]
\ds \quad \quad
+\sqrt{1-\rho^2}\sqrt{\left(\psi(\widetilde{Y}^N_{t_k})+\frac{\sigma
\psi'(\widetilde{Y}^N_{t_k})}{\delta_N}\int_{t_k}^{t_{k+1}}
(W_s-W_{t_k})ds\right) \vee \underline{\psi}}\,\, \Delta B_{k+1}.
\end{array}
\end{equation}

Note that in order to implement this scheme, one needs to simulate
both the Brownian increment $\Delta W_{k+1}$ and the random variable
$\int_{t_k}^{t_{k+1}} (W_s-W_{t_k}) ds$. This is straightforward as
one can easily check that

\[\left(\begin{array}{c}
\ds \Delta W_{k+1}\\[1mm]
\ds \int_{t_k}^{t_{k+1}} (W_s-W_{t_k}) ds \end{array}\right) \sim
\mathcal{N}\left(\left(\begin{array}{c}\ds
0\\[1mm]\ds 0\end{array}\right),\left(\begin{array}{cc}
\ds \delta_N&\ds \delta_N^2/2 \\[1mm]
\ds \delta_N^2/2 &\ds \delta_N^3/3\end{array}\right)\right).\]

We can now state our first main result :
\begin{theorem}
\label{THfort}Under the following assumptions
\begin{hypo}
\label{fbounded} $f$ and $\sigma$ are $\mathcal{C}^3$ functions,
$\frac{f}{\sigma}$ and $ff'$ are bounded
\end{hypo}
\begin{hypo}
\label{spos} $\underline{\psi} >0$
\end{hypo}
\begin{hypo}
 \label{hypTh1}
there exists a constant $K_{1}$ such that, $\forall (x,y) \in
\rr^2,$
\[\begin{array}{l}
 \ds \Big|(bh'+\frac{\sigma^{2}}{2}h'')(y)\Big|
 \leq K_{1} (1+|y|)\\[2mm]
 \ds \big|\sigma h'(y)\big|
 \leq K_{1} (1+|y|)\\[2mm]
\ds \Big|h(y)-h(x)\Big|\leq K_1 |y-x|
 \end{array}\]
\end{hypo}
\begin{hypo}
\label{hypTh2} there exists a constant $K_{2}$ such that, $\forall
(x,y) \in \rr^2,$
\[\begin{array}{l}
 \ds \Big|(b\psi'+\frac{\sigma^{2}}{2}\psi'')(y)\Big|
 \leq K_{2} (1+|y|)\\[2mm]
 \ds \big|\sigma \psi'(y)-\sigma \psi'(x)\big|
 \leq K_{2} |y-x|
 \end{array}\]
\end{hypo}
\begin{hypo}\label{Mil1}
$b$ and $\sigma$ are $\mathcal{C}^2$ functions with bounded first
and second derivatives
\end{hypo}

\begin{hypo}\label{Mil2}
there exists a positive constant $K$ such that $\forall (x,y) \in
\rr^2$
\[|\sigma\sigma'(x)-\sigma\sigma'(y)| \leq K |x-y|\]
\end{hypo}

the WeakTraj\_1 scheme has order one of weak trajectorial
convergence. More precisely, for each $p\geq 1$, there exists a
constant $C$ independent of the number of time steps $N$ such that
\[\ee\left[\max_{0\leq k\leq N}
\Big|\!\Big|\left(\widetilde{X}_{t_k},Y_{t_k}\right)-\left(\widetilde{X}^N_{t_k},\widetilde{Y}^N_{t_k}\right)\Big|\!\Big|^{2p}\right]
\leq \frac{C}{N^{2p}}.\]
\end{theorem}

The proof of the theorem relies on the order one of strong convergence of the
Milstein scheme (see Milstein \cite{Milstein} for the particular case $p=1$) :

\begin{lemma}
\label{MilsteinStrong}Under the assumptions
($\mathcal{H}$\ref{Mil1}) and ($\mathcal{H}$\ref{Mil2}), one has
that, $\forall p\geq1$, there exists a positive constant $C_p$
independent of $N$ such that
\[\ee\left(\max_{0\leq k \leq N}
\left|Y_{t_k}-\widetilde{Y}^N_{t_k}\right|^{2p}\right) \leq C_p
\delta_N^{2p}.\]
\end{lemma}

The proof for general $p$ is postponed to the appendix.

\begin{remark}
Before giving the proof of the theorem, we make a few comments on
its assumptions. ($\mathcal{H}$\ref{fbounded}) implies that $h$ and
$\psi$ are $\mathcal{C}^2$ functions which was implicitly assumed in
($\mathcal{H}$\ref{hypTh1}) and ($\mathcal{H}$\ref{hypTh2}). The
latter assumptions are expressed in a reduced form. One can check
that the following conditions on the coefficients of the original
SDE are sufficient for them to hold :
\begin{itemize}
\item $f$ and $\sigma$ are bounded $\mathcal{C}^4$ functions with
bounded derivatives.
\item $b$ is a bounded $\mathcal{C}^3$ function with
bounded derivatives.
\item $\exists \sigma_0 >0$ such that $\forall y \in \rr, \, \sigma(y) \geq
\sigma_0$.
\end{itemize}
\end{remark}

\begin{proofOFT}
Throughout the proof, we denote by $C$ a constant which can change
from one line to another while always being independent of $N$.
Thanks to Lemma \ref{MilsteinStrong}, we just have to control the
error on $\widetilde{X}$ :
\[\begin{array}{ll}
\ds \!\!\ee\!\!\left[\max_{0\leq k\leq N}
\!\!|\widetilde{X}_{t_k}-\widetilde{X}^N_{t_k}|^{2p}\right]\!\!&\!\!=\ds
\ee\!\!\left[\max_{0\leq k\leq N}
\Big|\rho(F(Y_{t_k})-F(\widetilde{Y}^N_{t_k}))+\sum_{j=0}^{k-1}\left(
\int_{t_{j}}^{t_{j+1}}
 \!\!h(Y_{s})ds-\delta_N
h(\widetilde{Y}^N_{t_j})\right.\right.\\[5mm]
&\ds\!\!\left. +\sqrt{\frac{1-\rho^{2}}{\delta_{N}}
\int_{t_{j}}^{t_{j+1}}\!\!\psi(Y_{s})ds}\,\Delta B_{j+1}\right.\\[4mm]
&\ds
\left.\left.-\sqrt{1-\rho^2}\sqrt{\left(\psi(\widetilde{Y}^N_{t_j})+\frac{\sigma
\psi'(\widetilde{Y}^N_{t_j})}{\delta_{N}}\int_{t_j}^{t_{j+1}}
\!\!(W_s-W_{t_j})ds\right)\vee \underline{\psi}}\,\,\,\Delta B_{j+1}\Big|^{2p}\right)\right]\\[4mm]
&\leq\ds 3^{2p-1}\, (\rho^{2p}I_0+I_{1}+(1-\rho^{2})^pI_{2})
\end{array}\]
where
\[I_0=\ee\left[\max_{0\leq k\leq N} \Big|F(Y_{t_k})-F(\widetilde{Y}^N_{t_k})\Big|^{2p}\right]\]
\[I_{1}=\ee\left[\max_{0\leq k\leq N} \Big|\sum_{j=0}^{k-1} \left( \int_{t_{j}}^{t_{j+1}}
 \!\!h(Y_{s})ds-\delta_N h(\widetilde{Y}^N_{t_j})\right)\Big|^{2p}\right]\]
and
\[\begin{array}{l}
 \ds I_{2}=\ee\left[\max_{0\leq k\leq N}\Big|\sum_{j=0}^{k-1} \left(\sqrt{\frac{1}{\delta_{N}} \int_{t_{j}}^{t_{j+1}}
   \!\!\psi(Y_{s})ds}\,\,-\right.\right.\\[5mm]
\ds \quad \quad \quad \quad  \left.\left.
\sqrt{\left(\psi(\widetilde{Y}^N_{t_j})+\frac{\sigma
\psi'(\widetilde{Y}^N_{t_j})}{\delta_{N}}\int_{t_j}^{t_{j+1}}
\!\!(W_s-W_{t_j})ds\right)\vee \underline{\psi}}\right) \Delta
B_{j+1}\Big|^{2p}\right].
\end{array}
\]

($\mathcal{H}$\ref{fbounded}) yields that $F$ is Lipschitz
continuous so using Lemma \ref{MilsteinStrong} we show that $I_0
\leq \frac{C}{N^{2p}}.$ Next, we have that
\[I_1 \leq C \left(\ee\left[\max_{0\leq k\leq N}\Big|\sum_{j=0}^{k-1} \int_{t_{j}}^{t_{j+1}}
 \!\!h(Y_{s})ds-\delta_N h(Y_{t_j})\Big|^{2p}\right]+\delta_N^{2p}\ee\left[\max_{0\leq k\leq N}\Big|\sum_{j=0}^{k-1}
  h(Y_{t_j})- h(\widetilde{Y}^N_{t_j})\Big|^{2p}\right]\right).\]
On one hand, thanks to assumption ($\mathcal{H}$\ref{hypTh1}) and
Lemma \ref{MilsteinStrong},
\[\delta_N^{2p} \ee\left[\max_{0\leq k\leq N}
\Big|\sum_{j=0}^{k-1}
  h(Y_{t_j})- h(\widetilde{Y}^N_{t_j})\Big|^{2p}\right] \leq
  C \delta_N \sum_{j=0}^{N-1}\ee\left[\Big|h(Y_{t_j})-
  h(\widetilde{Y}^N_{t_j})\Big|^{2p}\right]\leq\frac{C}{N^{2p}}.
  \]
On the other hand, using an integration by parts formula,

\[\!\!\begin{array}{rcl}
 \ds \overline{I}_{1}\!\!& \!\!:=\!\! &\!\! \ds \ee\left[\max_{0\leq k\leq N} \Big|\sum_{j=0}^{k-1} \int_{t_{j}}^{t_{j+1}}
 h(Y_{s})-h(Y_{t_j}) ds\Big|^{2p}\right]\\[5mm]
 \!\!&\!\!=\!\!&\!\!\ds  \ee\left[\max_{0\leq k\leq N} \Big|\sum_{j=0}^{k-1}
 \int_{t_{j}}^{t_{j+1}}
 (t_{j+1}-s)\left((bh'+\frac{\sigma^2h''}{2})(Y_s)ds+\sigma h'(Y_s)dW_s\right) \Big|^{2p}\right]\\[5mm]
 \!\!&\!\!\leq\!\!&\!\!\ds 2^{2p-1}\left(\ee\left[\max_{0\leq k\leq N} \Big| \int_{0}^{t_{k}}\!\!
 (\tau_s-s)(bh'+\frac{\sigma^2h''}{2})(Y_s)ds \Big|^{2p}\right]\!\!+\!\ee\left[\max_{0\leq k\leq N} \Big| \int_{0}^{t_{k}}\!\!
 (\tau_s-s)\sigma h'(Y_s)dW_s \Big|^{2p}\right]\right)\\[5mm]
 \end{array}
\]
where we denoted by $\tau_s$ the lowest discretization point greater
than $s$ : $\tau_s = \lceil \frac{s}{\delta_N} \rceil \delta_N$.
Using Jensen's inequality for the first integral and the
Burkholder-Davis-Gundy inequality for the second, we obtain
\[\begin{array}{rcl}
 \ds \overline{I}_{1}& \leq&\ds C\left(\ee\left[\max_{0\leq k\leq N} t_k^{2p-1} \int_{0}^{t_{k}}
 (\tau_s-s)^{2p}\left|(bh'+\frac{\sigma^2h''}{2})(Y_s)\right|^{2p}ds \right]\right.\\[5mm]
 &&\ds \quad \quad \quad \quad \left.+\ee\left[\left(\int_{0}^{T}
 (\tau_s-s)^2 \left|\sigma h'(Y_s)\right|^2ds\right)^p \right]\right)\\[5mm]
 &\leq&\ds \frac{C}{N^{2p}} \int_0^T
 \ee\left[\left|(bh'+\frac{\sigma^2h''}{2})(Y_s)\right|^{2p}+\left|\sigma
 h'(Y_s)\right|^{2p}\right]ds.
\end{array}\]
Under the assumptions of Lemma \ref{MilsteinStrong}, $\sup_{0\leq t\leq T}
\ee(|Y_s|^{2p})<\infty$ (see Problem 3.15 p.~306 of Karatzas and Shreve
\cite{KaratzasShreve} for example) so, with the help of assumption
($\mathcal{H}$\ref{hypTh1}), we conclude that $\overline{I}_{1} \leq
\frac{C}{N^{2p}}$ and hence $I_{1} \leq \frac{C}{N^{2p}}$. We now turn to the
last term. Using the Burkholder-Davis-Gundy inequality, we get
\begin{equation}\label{Eq*}
\begin{array}{rcl}
\!\!\ds I_{2} \!\!&\!\!\leq\!\! &\!\! \ds C \delta_{N}^p
\ee\!\left[\!\left(\!\!\sum_{j=0}^{N-1}\!\left(\!\!\sqrt{\frac{1}{\delta_{N}}
\int_{t_{j}}^{t_{j+1}}
   \psi(Y_{s})ds}\!-\!\!\sqrt{\left(\psi(\widetilde{Y}^N_{t_j})+\frac{\sigma
\psi'(\widetilde{Y}^N_{t_j})}{\delta_{N}}\int_{t_j}^{t_{j+1}}
\!\!(W_s-W_{t_j})ds\right)\vee
\underline{\psi}}\right)^{\!\!2}\right)^{\!\!p}\right]\\[8mm]
\!\!&\!\!\leq\!\!&\!\!\ds \delta_N
\sum_{j=0}^{N-1}\ee\left[\left|\sqrt{\frac{1}{\delta_{N}}
\int_{t_{j}}^{t_{j+1}}
   \psi(Y_{s})ds}-\sqrt{\left(\psi(\widetilde{Y}^N_{t_j})+\frac{\sigma
\psi'(\widetilde{Y}^N_{t_j})}{\delta_{N}}\int_{t_j}^{t_{j+1}}
\!\!(W_s-W_{t_j})ds\right)\vee \underline{\psi}}\right|^{2p}\right].
\end{array}
\end{equation}
Assumption ($\mathcal{H}$\ref{spos}) yields that the two terms
appearing in the square root are bounded from below by
$\underline{\psi} >0$ so we have that
\[\begin{array}{rcl}
 \ds I_{2}&\leq&\ds C \delta_{N}  \sum_{j=0}^{N-1} \ee\left[\left|{\frac{1}{\delta_{N}} \int_{t_{j}}^{t_{j+1}}
   \psi(Y_{s})ds}-{\left(\psi(\widetilde{Y}^N_{t_j})+\frac{\sigma
\psi'(\widetilde{Y}^N_{t_j})}{\delta_{N}}\int_{t_j}^{t_{j+1}}
(W_s-W_{t_j})ds\right)\vee \underline{\psi}}\right|^{2p}\right]\\[5mm]
&\leq&\ds C N^{2p-1} \sum_{j=0}^{N-1}
\ee\left[\left|{\int_{t_{j}}^{t_{j+1}}
   \psi(Y_{s})ds}-{\left(\psi(\widetilde{Y}^N_{t_j}) \delta_{N}+\sigma
\psi'(\widetilde{Y}^N_{t_j})\int_{t_j}^{t_{j+1}}
(W_s-W_{t_j})ds\right)}\right|^{2p}\right]\\[5mm]
&\leq&\ds C N^{2p-1}\sum_{j=0}^{N-1}
\left(\overline{I}_2^j+\widetilde{I}_2^j\right)
\end{array}
\]
where
\begin{equation}
\label{termIj2}
\overline{I}_2^j=\ee\left[\left|{\int_{t_{j}}^{t_{j+1}}
   \psi(Y_{s})ds}-{\left(\psi(Y_{t_j}) \delta_{N}+\sigma
\psi'(Y_{t_j})\int_{t_j}^{t_{j+1}}
(W_s-W_{t_j})ds\right)}\right|^{2p}\right]
\end{equation} and
\[\widetilde{I}_2^j=\ee\left[\left| \delta_{N}\left(\psi(Y_{t_j})-\psi(\widetilde{Y}^N_{t_j})\right) +\left(\sigma
\psi'(Y_{t_j})-\sigma
\psi'(\widetilde{Y}^N_{t_j})\right)\int_{t_j}^{t_{j+1}}
(W_s-W_{t_j})ds\right|^{2p}\right].\]

Again, integrating by parts yields that
\[\overline{I}_2^j=\ee\left[\left|\int_{t_{j}}^{t_{j+1}}
(t_{j+1}-s)\left((\sigma\psi'(Y_s)-\sigma
\psi'(Y_{t_j}))dW_s+((b\psi'+\frac{\sigma^2}{2}\psi'')(Y_s))ds\right)\right|^{2p}\right]\]

We control the stochastic integral term as follows
\[\!\!\begin{array}{rcl}
\ds \ee\!\!\left[\left|\int_{t_{j}}^{t_{j+1}}\!\!
(t_{j+1}-s)(\sigma\psi'(Y_s)-\sigma\psi'(Y_{t_j}))dW_s\right|^{2p}\right]\!\!\!\!&\!\!\leq\!\!&\!\!\!\!\ds
C \delta_N^{p-1}\ee\!\!\left[\int_{t_{j}}^{t_{j+1}}\!\!
(t_{j+1}-s)^{2p}|\sigma\psi'(Y_s)-\sigma\psi'(Y_{t_j})|^{2p}ds\right]\\[3mm]
\!\!\!\!&\!\!\leq\!\!&\!\!\!\!\ds C \delta_N^{3p-1} \int_{t_{j}}^{t_{j+1}} \ee\left[\left|\sigma\psi'(Y_s)-\sigma\psi'(Y_{t_j})\right|^{2p}\right]ds\\[3mm]
\!\!\!\!&\!\!\leq\!\!&\!\!\!\!\ds C \delta_N^{3p-1} \int_{t_{j}}^{t_{j+1}} \ee\left[\left|Y_s-Y_{t_j}\right|^{2p}\right]ds\\[3mm]
\!\!\!\!&\!\!\leq\!\!&\!\!\!\!\ds C \delta_N^{3p-1} \int_{t_{j}}^{t_{j+1}} \left|s-t_j\right|^pds\\[3mm]
\!\!\!\!&\!\!\leq\!\!&\!\!\!\!\ds C \delta_N^{4p}.
\end{array}\]
The third inequality is due to assumption ($\mathcal{H}$\ref{hypTh2}) and the
fourth one is a standard result on the control of the moments of the increments
of the solution of a SDE with Lipschitz continuous coefficients (see Problem
3.15 p. 306 of Karatzas and Shreve \cite{KaratzasShreve} for example).

We also control the other term thanks to assumption
($\mathcal{H}$\ref{hypTh2}) :
\[\begin{array}{rcl}
\ds \ee\left[\left|\int_{t_{j}}^{t_{j+1}}
(t_{j+1}-s)(b\psi'+\frac{\sigma^2}{2}\psi'')(Y_s)ds\right|^{2p}\right]\!\!&\leq&\!\!\ds
\delta_N^{2p-1}\ee\left[\int_{t_{j}}^{t_{j+1}}
(t_{j+1}-s)^{2p}|(b\psi'+\frac{\sigma^2}{2}\psi'')(Y_s)|^{2p}ds\right]\\[3mm]
\!\!&\leq&\!\!\ds  \delta_N^{4p-1} \int_{t_{j}}^{t_{j+1}} \ee\left[\left|(b\psi'+\frac{\sigma^2}{2}\psi'')(Y_s)\right|^{2p}\right]ds\\[3mm]
\!\!&\leq&\!\!\ds C \delta_N^{4p}.
\end{array}\]
Hence, $\overline{I}^j_{2} \leq \frac{C}{N^{4p}}$. To conclude the
proof of the theorem, it remains to show a similar result for
$\widetilde{I}^j_{2}$ :
\[\begin{array}{rcl}
\ds \widetilde{I}^j_{2} &\leq&\ds 2^{2p-1} \ee\left[\left|
\delta_{N}\left(\psi(Y_{t_j})-\psi(\widetilde{Y}^N_{t_j})\right)\right|^{2p}
+\left|\left(\sigma \psi'(Y_{t_j})-\sigma
\psi'(\widetilde{Y}^N_{t_j})\right)\int_{t_j}^{t_{j+1}}
(W_s-W_{t_j})ds\right|^{2p}\right]\\[3mm]
&\leq&\ds C\left(\delta_N^{2p}
\ee\left[\left|Y_{t_j}-\widetilde{Y}^N_{t_j}\right|^{2p}\right]+\frac{\delta_N^{3p}}{3^p}\ee\left[\left|Y_{t_j}-\widetilde{Y}^N_{t_j}\right|^{2p}\right]
\right)\\[3mm]
&\leq&\ds \frac{C}{N^{4p}}.
\end{array}\] The second inequality is due to the fact that $\psi$
is Lipschitz continuous (thanks to assumption
($\mathcal{H}$\ref{fbounded})) for the first term and to the
independence of $\left(\sigma \psi'(Y_{t_j})-\sigma
\psi'(\widetilde{Y}^N_{t_j})\right)$ and $\int_{t_j}^{t_{j+1}}
(W_s-W_{t_j})ds$ for the second term.
\end{proofOFT}

\begin{remark}
\label{MLMC} Our scheme exhibits the same convergence properties as the
Cruzeiro \textit{et al.} \cite{CruzeiroMalliavinThalmaier} scheme. In addition
to the fact that it involves fewer terms, it presents the advantage of
improving the multilevel Monte Carlo convergence. This method, which is a
generalization of the statistical Romberg extrapolation method of Kebaier
\cite{Kebaier}, was introduced by Giles \cite{Giles1,Giles2}.

Indeed, consider the discretization scheme with time step
$\delta_{2N}=\frac{T}{2N}$ :
\[\begin{array}{l} \ds
\forall 0\leq k\leq 2N-1,\,\,
\widetilde{X}_{{\frac{(k+1)T}{2N}}}^{2N}=\widetilde{X}_{{\frac{kT}{2N}}}^{2N}+\rho
\left(F(\widetilde{Y}^{2N}_{\frac{(k+1)T}{2N}})-F(\widetilde{Y}^{2N}_{{\frac{kT}{2N}}})\right)+ \delta_{2N} h(\widetilde{Y}^{2N}_{{\frac{kT}{2N}}})+\sqrt{1-\rho^2}\\[4mm]
\ds \quad\quad \quad\quad \quad\quad \quad\quad \quad
\times\sqrt{\left(\psi(\widetilde{Y}^{2N}_{{\frac{kT}{2N}}})+\frac{\sigma
\psi'(\widetilde{Y}^{2N}_{{\frac{kT}{2N}}})}{\delta_{2N}}\int_{{\frac{kT}{2N}}}^{{\frac{(k+1)T}{2N}}}
(W_s-W_{{\frac{kT}{2N}}})ds\right) \vee \underline{\psi}}\,\,
\left(B_{{\frac{(k+1)T}{2N}}}-B_{\frac{kT}{2N}}\right).
\end{array}\]
Denote by
$v^{2N}_k=\sqrt{1-\rho^2}\sqrt{\left(\psi(\widetilde{Y}^{2N}_{{\frac{kT}{2N}}})+\frac{\sigma
\psi'(\widetilde{Y}^{2N}_{{\frac{kT}{2N}}})}{\delta_{2N}}\int_{{\frac{kT}{2N}}}^{{\frac{(k+1)T}{2N}}}
(W_s-W_{{\frac{kT}{2N}}})ds\right) \vee \underline{\psi}}$ the
random variable which multiplies the increment of the Brownian
motion $\left(B_{{\frac{(k+1)T}{2N}}}-B_{\frac{kT}{2N}}\right)$.
Because of the independence properties,
$\left(\widetilde{X}^N_{t_k}\right)_{0\leq k\leq N}$ has the same
distribution law as the vector
$\left(\widetilde{\widetilde{X}}^N_{t_k}\right)_{0\leq k\leq N}$
defined inductively by $\widetilde{\widetilde{X}}^N_{t_0}=\log(s_0)$
and
\[\begin{array}{l} \ds
\forall 0\leq k\leq N-1,\,\,
\widetilde{\widetilde{X}}_{t_{k+1}}^{N}=\widetilde{\widetilde{X}}_{t_{k}}^{N}+\rho
\left(F(\widetilde{Y}^{N}_{t_{k+1}})-F(\widetilde{Y}^{N}_{t_{k}})\right)+ \delta_{N} h(\widetilde{Y}^{N}_{t_k})\\[4mm]
\ds \quad \quad\quad \quad\quad \quad\quad \quad\quad \quad
+\sqrt{1-\rho^2}\sqrt{\left(\psi(\widetilde{Y}^{N}_{t_k})+\frac{\sigma
\psi'(\widetilde{Y}^{N}_{t_k})}{\delta_{N}}\int_{t_k}^{t_{k+1}}
(W_s-W_{t_k})ds\right) \vee \underline{\psi}}\,\, \Delta
\widetilde{B}^{N}_{k+1}
\end{array}\]
where
\begin{equation}
\label{couplage} \Delta
\widetilde{B}^{N}_{k+1}=\sqrt{2}\left[\frac{v^{2N}_{2k}\left(B_{{\frac{(2k+1)T}{2N}}}-B_{\frac{2kT}{2N}}\right)+v^{2N}_{2k+1}
\left(B_{{\frac{(2k+2)T}{2N}}}-B_{\frac{(2k+1)T}{2N}}\right)}{\sqrt{\left(v^{2N}_{2k}\right)^2+\left(v^{2N}_{2k+1}\right)^2}}\right].
\end{equation}

Going over the proof of the theorem, one can show in the same way
that
\begin{equation}
\label{IllusStrong} \ee\left[\max_{0\leq k \leq N}
\left|\widetilde{\widetilde{X}}^N_{t_k}-\widetilde{X}^{2N}_{t_{k}}\right|^2\right]
= \mathcal{O}(N^{-2}).
\end{equation} Hence, one can apply the
multilevel Monte Carlo method to compute the expectation of a
Lipschitz continuous functional of $X$ and reduce the computational
cost to achieve a desired root-mean-square error of $\epsilon >0$ to
a $\mathcal{O}(\epsilon^{-2})$.

To summarize, the particular structure of our scheme enabled us to reconstruct
the coupling which allows to efficiently control the error between the scheme
with time step $\frac{T}{N}$ and the one with time step $\frac{T}{2N}$. This
does not seem possible with the Cruzeiro \textit{et al.}
\cite{CruzeiroMalliavinThalmaier} scheme.

\end{remark}

{From} a practical point of view, it is more interesting to obtain a
convergence result for the stock price. It is also more challenging
because the exponential function is not globally Lipschitz
continuous. We can nevertheless state the following corollary with
some general assumptions and we will see in the next section that we
can make them more precise in the case where $(Y_t)_{t\in[0,T]}$ is an
Ornstein-Uhlenbeck process.

\begin{Corollary}
\label{CorolStock}Let $p\geq1$. Under the assumptions of Theorem
\ref{THfort} and if
\begin{hypo}
\label{hypStock} $\ds \exists \epsilon >0 \text{ such that
}\ee\left[\max_{0\leq k \leq N} S_{t_k}^{2p+\epsilon}\right]
+\ee\left[\max_{0\leq k \leq N} e^{(2p+\epsilon)
\widetilde{X}^N_{t_k}}\right]<\infty$
\end{hypo}
then there exists a positive constant $C$ independent of $N$ such
that
\[\ee\left[\max_{0\leq k\leq N}
\left|e^{\widetilde{X}_{t_k}}-e^{\widetilde{X}^N_{t_k}}\right|^{2p}\right]
\leq \frac{C}{N^{2p}}.\]
\end{Corollary}
\begin{proof}
Using Hölder inequality we have that
\[\begin{array}{rcl} \ds \ee\left[\max_{0\leq k\leq N}
\left|e^{\widetilde{X}_{t_k}}-e^{\widetilde{X}^N_{t_k}}\right|^{2p}\right]
&\leq&\ds \ee\left[\max_{0\leq k\leq N}\left(e^{2p
\widetilde{X}_{t_k}} \vee e^{2p \widetilde{X}^N_{t_k}}\right)
\left|\widetilde{X}_{t_k}-\widetilde{X}^N_{t_k}\right|^{2p}\right]
\\[5mm]
&\leq&\ds \left(\ee\left[\max_{0\leq k\leq N}
S_{t_k}^{2p+\epsilon}\right]+\ee\left[\max_{0\leq k\leq N}
e^{(2p+\epsilon)
\widetilde{X}^N_{t_k}}\right]\right)^{\frac{2p}{2p+\epsilon}}\\[5mm]
&&\ds\quad\quad\quad\times \left(\ee\left[\max_{0\leq k\leq N}
\left|\widetilde{X}_{t_k}-\widetilde{X}^N_{t_k}\right|^{\frac{2p\epsilon+4p^2}{\epsilon}}\right]\right)^{\frac{\epsilon}{2p+\epsilon}}.
\end{array}
\]
We conclude by assumption ($\mathcal{H}$\ref{hypStock}) and Theorem
\ref{THfort}.
\end{proof}

\begin{remark}
Had we introduced a new cut-off to our scheme as follows
\[\begin{array}{l} \ds
\widetilde{X}_{t_{k+1}}^N=\widetilde{X}_{t_{k}}^N+\rho
\left(F(\widetilde{Y}^N_{t_{k+1}})-F(\widetilde{Y}^N_{t_{k}})\right)+ \delta_N h(\widetilde{Y}^N_{t_k})\\[4mm]
\ds \quad \quad
+\sqrt{1-\rho^2}\sqrt{\left(\psi(\widetilde{Y}^N_{t_k})+\frac{\sigma
\psi'(\widetilde{Y}^N_{t_k})}{\delta_N}\int_{t_k}^{t_{k+1}}
(W_s-W_{t_k})ds\right)\wedge \overline{\psi} \vee
\underline{\psi}}\,\, \Delta B_{k+1}
\end{array}\]
assumption ($\mathcal{H}$\ref{hypStock}) would have been induced by
assuming that the functions $F,f$ and $h$ are bounded.
\end{remark}

\subsection{Special case of an Ornstein-Uhlenbeck process driving the
volatility}

For many stochastic volatility models, the process
$(Y_t)_{t\in[0,T]}$ which drives the volatility is an
Ornstein-Uhlenbeck process. For example, this is the case for all
the models cited in the introduction but the Heston model.
Therefore, it is useful to focus on this particular case. We will
hereafter suppose that $(Y_t)_{t\in[0,T]}$ is the solution of the
following SDE
\begin{equation}
\label{OU} dY_t=\nu dW_t+\kappa(\theta-Y_t)dt,\;Y_0=y_0
\end{equation}
with $\nu>0$ and $\kappa,\theta\in\rr$. Since exact simulation is
possible, we can replace the Milstein discretization by the true
solution in our previous scheme :

$\,$

\framebox{\textbf{WeakTraj\_1 scheme when Y is an O-U process}}
\begin{equation}
\label{strong_Y-OU} \begin{array}{l} \ds
\widetilde{X}_{t_{k+1}}^N=\widetilde{X}_{t_{k}}^N+\rho
\left(F(Y_{t_{k+1}})-F(Y_{t_{k}})\right)+ \delta_N h(Y_{t_k})\\[4mm]
\ds \quad \quad +\sqrt{1-\rho^2}\sqrt{\left(\psi(Y_{t_k})+\frac{\nu
\psi'(Y_{t_k})}{\delta_N}\int_{t_k}^{t_{k+1}} (W_s-W_{t_k})ds\right)
\vee \underline{\psi}}\,\, \Delta B_{k+1}.
\end{array}
\end{equation}
Note that we require the exact simulation of both
$(Y_{t_k},Y_{t_{k+1}})$ and $\int_{t_k}^{t_{k+1}} (W_s-W_{t_k}) ds$.
The unique solution of (\ref{OU}) is $Y_t=y_0e^{-\kappa
t}+\theta(1-e^{-\kappa t})+\nu\int_0^t e^{-\kappa (t-s)} dW_s$ and
one can easily deduce that, $\forall k \in\{0,\dots,N-1\}$,
\[\left(\begin{array}{c}
Y_{t_{k+1}}-e^{-\kappa \delta_N}Y_{t_k}\\[3mm]
\int_{t_k}^{t_{k+1}} (W_s-W_{t_k}) ds
\end{array}\right)\sim
\mathcal{N}\left(M,\Gamma\right)\] where $M=\left(\begin{array}{c}
\theta(1-e^{-\kappa
\delta_N})\\[3mm]0\end{array}\right)$ and $\Gamma=\left(\begin{array}{cc}
\frac{\nu^2}{2\kappa}(1-e^{-2\kappa \delta_N})&\frac{\nu}{\kappa^2} (1-e^{-\kappa \delta_N}(1+\kappa \delta_N))\\[3mm]
\frac{\nu}{\kappa^2} (1-e^{-\kappa \delta_N}(1+\kappa \delta_N))&
\frac{\delta_N^3}{3}\end{array}\right)$.

\vspace{2mm} We first state the following technical lemma whose
proof is postponed to the appendix :

\begin{lemma}\label{MomentExpo}
$\forall\, c_1>0, c_2 \in[0,1)$,
\[\ee\left(e^{c_1 \sup_{0\leq t \leq T} |Y_t|^{1+c_2}}\right) <
\infty.\]

\begin{equation}\label{contprobzero}
\mbox{Moreover, when }y_0\neq 0,\;\forall \alpha>0,\;\exists
C<+\infty,\;\forall N\in\nn^*,\;\sup_{t\in[0,T]}\pp\left[|Y_t|\leq
N^{-\alpha}\right]\leq CN^{-\alpha}.
\end{equation}
\end{lemma}

As might be expected, it is possible to weaken the assumptions of
Theorem \ref{THfort}. In particular, we relax the assumption on the
lower bound of the volatility ($\mathcal{H}$\ref{spos}) and replace
it with a weaker one (see assumption ($\mathcal{H}$\ref{hypOU})
below). The following theorem applies for Scott's
    model \cite{Scott} (and therefore for the Hull and White
   \cite{HullWhite} model) where we have
   $h(y)=r-\frac{\sigma_0^2e^{2y}}{2}-\rho
   \sigma_0e^y(\frac{\kappa}{\nu}(\theta-y)+\frac{\nu}{2})$ and
   $\psi(y)=\sigma_0^2e^{2y}$.
\begin{theorem}
\label{THfort_Y-OU}Let $p\geq 1$. Suppose that $Y$ is solution of
(\ref{OU}) and that the scheme is defined by (\ref{strong_Y-OU}).
Under assumption ($\mathcal{H}$\ref{spos}) of Theorem \ref{THfort}
and if
\begin{hypo}
\label{fC3} $f$ is  a $\mathcal{C}^3$ function
\end{hypo}
\begin{hypo}
\label{hypTh1_YOU} there exist three constants $c_0>0, c_1>0$ and
$c_2 \in[0,1)$ such that, $\forall y \in \rr,$
\[\begin{array}{l}
 \ds \Big|\kappa(\theta-y)h'(y)+\frac{\nu^{2}}{2}h''(y)\Big|
 \leq c_{0} e^{c_1|y|^{1+c_2}}\\[2mm]
 \ds \big|h'(y)\big|
 \leq c_{0} e^{c_1|y|^{1+c_2}}\\[1mm]
   \ds \Big|\kappa(\theta-y)\psi'(y)+\frac{\nu^{2}}{2}\psi''(y)\Big|
 \leq c_{0} e^{c_1|y|^{1+c_2}}\\[2mm]
 \ds \big|\psi''(y)\big|
 \leq c_{0} e^{c_1|y|^{1+c_2}}
 \end{array}\]

\end{hypo}

then there exists a constant $C$ independent of the number of time
steps $N$ such that
\[\ee\left[\max_{0\leq k\leq N}
\Big|\widetilde{X}_{t_k}-\widetilde{X}^N_{t_k}\Big|^{2p}\right] \leq
\frac{C}{N^{2p}}\]

The same result holds true when we replace assumption
($\mathcal{H}$\ref{spos}) by
\begin{hypo}
\label{hypOU}
\[\begin{array}{rcl}
\ds \forall y\in \rr,\;\psi(y)&>&\ds 0\\[3mm]
\ds \sup_{t\leq T}\ee\left(\frac{1}{\psi^{p(1+\epsilon)}(Y_t)}\right)&<&\ds \infty.\\
\end{array}\]
\end{hypo}
\end{theorem}

\begin{proof}
The proof of the first part of the theorem repeats the proof of
Theorem \ref{THfort} with fewer terms to control because of the
exact simulation of $(Y_t)_{t\in[0,T]}$. At the places where we used
assumptions ($\mathcal{H}$\ref{hypTh1}) and
($\mathcal{H}$\ref{hypTh2}), we use assumption
($\mathcal{H}$\ref{hypTh1_YOU}) together with Lemma
\ref{MomentExpo}.

We now focus on the second part of the theorem. According to
equation (\ref{Eq*}), all we have to show is the existence of a
positive constant $C$ independent of $N$ such that $\forall j \in
\{0,\dots, N-1\}$
\[\ee\left[\left|\sqrt{A_j}-\sqrt{D_j}\right|^{2p}\right] \leq
\frac{C}{N^{2p}}\] where
\begin{itemize}
\item[-] $A_j = \frac{1}{\delta_{N}} \int_{t_{j}}^{t_{j+1}}\psi(Y_{s})ds$,
\item[-] $D_j=\left(\psi(Y_{t_j})+\frac{\nu
\psi'(Y_{t_j})}{\delta_{N}}\int_{t_j}^{t_{j+1}}
(W_s-W_{t_j})ds\right)\vee \underline{\psi}$.
\end{itemize}
One has
\begin{align*}
  |\sqrt{A_j}-\sqrt{D_j}|^{2p}=\frac{|A_j-D_j|^{2p}}{(\sqrt{A_j}+\sqrt{D_j})^{2p}}\leq \frac{|A_j-D_j|^{2p}}{A_j^p}.
\end{align*}
Hence by H\"older's and Jensen's inequalities,
\begin{align*}
 \ee \left[|\sqrt{A_j}-\sqrt{D_j}|^{2p}\right]&\leq \left(\ee\left[\frac{1}{A_j^{p(1+\epsilon)}}\right]\right)^{\frac{1}{1+\epsilon}}
\left(\ee\left[|A_j-D_j|^{2p\frac{1+\epsilon}{\epsilon}}\right]\right)^{\frac{\epsilon}{1+\epsilon}}\\&\leq
\left(\frac{1}{\delta_N} \int_{t_j}^{t_{j+1}}\!\!
\ee\left(\frac{1}{\psi^{p(1+\epsilon)}(Y_s)}\right)
ds\right)^{\frac{1}{1+\epsilon}}\left(\ee\left[|A_j-D_j|^{2p\frac{1+\epsilon}{\epsilon}}\right]\right)^{\frac{\epsilon}{1+\epsilon}}.
\end{align*}

Thanks to assumption ($\mathcal{H}$\ref{hypOU}), the first term in
the right-hand-side is smaller than a finite constant not depending
on $N$ whereas the second term is smaller than $C\delta_N^{2p}$
according to the estimation of $\overline{I}^j_2$ in the proof of
Theorem \ref{THfort}.\end{proof} The following proposition is
dedicated to the Stein and Stein
   \cite{SteinStein} and the quadratic Gaussian models which satisfy neither
   assumption ($\mathcal{H}$\ref{spos}) nor assumption ($\mathcal{H}$\ref{hypOU}) since $\psi(y)$ vanishes at the origin.
\begin{proposition}
\label{SSqG}
  Suppose that $Y$ is solution of
(\ref{OU}) starting from $y_0$ and that the scheme is defined
by (\ref{strong_Y-OU}). Then for $p\geq 1$ and $\varepsilon>0$,
there is a constant $C$ not depending on $N$ such that
\[\ee\left[\max_{0\leq k\leq N}
\Big|\widetilde{X}_{t_k}-\widetilde{X}^N_{t_k}\Big|^{2p}\right] \leq
\begin{cases}
  \frac{C}{N^{\frac{2p+1}{2}-\varepsilon}}\mbox{ in the Stein and Stein
   model when $y_0\neq 0$}\\
\frac{C}{N^{2p-\varepsilon}}\mbox{ in the quadratic
Gaussian model}
\end{cases}.\]
In particular, the order of convergence in $L^2$ is not smaller that
$3/4-\varepsilon$ in the Stein and Stein model.
\end{proposition}
\begin{proof}
In both the Stein and Stein and the quadratic Gaussian models the
function $h(y)$, respectively equal to
$r-\frac{y^2}{2}-\rho\left(\frac{\kappa
y}{\nu}(\theta-y)+\frac{\nu}{2}\right)$ and
$r-\frac{y^4}{2}-\rho\left(\frac{\kappa y^2}{\nu}(\theta-y)+\nu
y\right)$ satisfies the first two inequalities of assumption
($\mathcal{H}$\ref{hypTh1_YOU}). So we only need to focus on
$\ee\left[\left|\sqrt{A_j}-\sqrt{D_j}\right|^{2p}\right]$ with
$$(A_j,D_j,\psi(y))=\begin{cases}\left(\frac{1}{\delta_{N}} \int_{t_{j}}^{t_{j+1}}Y^2_{s}ds,\left(Y^2_{t_j}+\frac{2\nu
Y_{t_j}}{\delta_{N}}\int_{t_j}^{t_{j+1}} (W_s-W_{t_j})ds\right)^+,y^2\right)\mbox{ in the Stein and Stein model}\\
\left(\frac{1}{\delta_{N}} \int_{t_{j}}^{t_{j+1}}Y^4_{s}ds,\left(Y^4_{t_j}+\frac{4\nu
Y^3_{t_j}}{\delta_{N}}\int_{t_j}^{t_{j+1}} (W_s-W_{t_j})ds\right)^+,y^4\right)\mbox{ in the quadratic Gaussian model}
\end{cases}.
$$
So in both cases $D_j\leq \frac{\psi(Y_{t_j})}{2}\Rightarrow
\left|\frac{4\nu }{\delta_{N}}\int_{t_j}^{t_{j+1}}
(W_s-W_{t_j})ds\right|\geq \frac{|Y_{t_j}|}{2}$. In what follows, we
use repeatedly that the power function $\psi$ is non-decreasing on
the positive half line and commutes with other power functions. Let
$(\alpha_l)_{1\leq l\leq L}$ be a decreasing sequence in
$(0,\frac{1}{2})$. Using the convention $N^{-\alpha_{L+1}}=+\infty$,
one has
\begin{equation}
\begin{array}{rcl}
\ds \ee\left[\left|\sqrt{A_j}-\sqrt{D_j}\right|^{2p}\right]& \leq &\ds \ee\left[\left|\sqrt{A_j}-\sqrt{D_j}\right|^{2p}\ind_{\{|Y_{t_j}|<N^{-\alpha_{1}}\}}\right]\\
&&\ds +\sum_{l=1}^L\ee\left[\frac{\left|A_j-D_j\right|^{2p}}{(\sqrt{A_j}+\sqrt{D_j})^{2p}}
\ind_{\{N^{-\alpha_{l}}\leq |Y_{t_j}|<N^{-\alpha_{l+1}},D_j\geq \frac{\psi(Y_{t_j})}{2}\}}\right]\\
&&\ds +\sum_{l=1}^L\ee\left[\left|\sqrt{A_j}-\sqrt{D_j}\right|^{2p}\ind_{\{N^{-\alpha_{l}}\leq |Y_{t_j}|<N^{-\alpha_{l+1}},\left|\frac{4\nu
}{\delta_{N}}\int_{t_j}^{t_{j+1}} (W_s-W_{t_j})ds\right|\geq \frac{|Y_{t_j}|}{2}\}}\right]\\
&\leq &\ds \ee\left[(A_j^p+D^p_j)\ind_{\{|Y_{t_j}|<N^{-\alpha_{1}}\}}\right]+
C\sum_{l=1}^L\psi(N^{p\alpha_{l}})\ee\left[\left|A_j-D_j\right|^{2p}\ind_{\{|Y_{t_j}|<N^{-\alpha_{l+1}}\}}\right]\\
&&+\ee\left[(A_j^p+D^p_j)\ind_{\{|\frac{4\nu}{\delta_{N}}\int_{t_j}^{t_{j+1}}
(W_s-W_{t_j})ds|>\frac{N^{-\alpha_{1}}}{2}\}}\right].
\end{array}
\label{decompsteinstein}
\end{equation}
Since conditionally on $Y_{t_j}$, for $s\in[t_j,t_{j+1}]$,
$Y_s\sim{\mathcal N}_1(Y_{t_j}e^{-\kappa
(s-t_j)}+\theta(1-e^{-\kappa
(s-t_j)}),\frac{\nu^2}{2\kappa}(1-e^{-2\kappa (s-t_j)}))$
(convention : $\frac{\nu^2}{2\kappa}(1-e^{-2\kappa
(s-t_j)})=\nu^2(s-t_j)$ when $\kappa=0$) and
$\frac{1}{\delta_N}\int_{t_j}^{t_{j+1}} (W_s-W_{t_j})ds\sim{\mathcal
N}_1(0,\frac{\delta_N}{3})$ is independent from $Y_{t_j}$ the first
term in the right-hand-side of \eqref{decompsteinstein} is not
greater than
\begin{align*}
&\ee\!\left[\!\frac{1}{\delta_N}\!\int_{t_j}^{t_{j+1}}\!\!\!\!\ee\left[\psi(Y_s)^p|Y_{t_j}\right]ds
\ind_{\{|Y_{t_j}|<N^{-\alpha_{1}}\}}\!\right]\!\!+\!2^{p-1}\ee\!\left[\!\left(\psi(Y_{t_j})^p\!+\!\left|\frac{\nu
\psi'(Y_{t_j})}{\delta_{N}}\int_{t_j}^{t_{j+1}}\!\!\!\!
(W_s-W_{t_j})ds\right|^p\right)\ind_{\{|Y_{t_j}|<N^{-\alpha_{1}}\}}\right]\\
&\leq
C\ee\left[(\psi(Y_{t_j})^p+\psi(\delta_N^\frac{p}{2}))\ind_{\{|Y_{t_j}|<N^{-\alpha_{1}}\}}\right]+
C(\psi(N^{-p\alpha_1})+\psi'(N^{-p\alpha_1})N^{-\frac{p}{2}})\pp(|Y_{t_j}|<N^{-\alpha_{1}})\\&\leq
C\psi(N^{-p\alpha_1})\pp(|Y_{t_j}|<N^{-\alpha_{1}}),
\end{align*}
where we used $\alpha_1<\frac{1}{2}$ for
the last inequality.
By Hölder's inequality and since
$\alpha_1<\frac{1}{2}$ and $G=\sqrt{\frac{3}{\delta^3_N}}\int_{t_j}^{t_{j+1}}
(W_s-W_{t_j})ds\sim{\mathcal N}_1(0,1)$, the third term in the
right-hand-side of \eqref{decompsteinstein} is not greater than a constant multiplying
\begin{align*}
\ee^{\frac{1}{1+\xi}}\!\!\left[\!\frac{1}{\delta_N}\!\!\int_{t_j}^{t_{j+1}}\!\!\!\!\psi(Y_s)^{p(1+\xi)}ds+
\psi(Y_{t_j})^{p(1+\xi)}+\frac{|\psi'(Y_{t_j})|^{p(1+\xi)}}{\delta_N}
\int_{t_j}^{t_{j+1}}\!\!\!\!\psi((s-t_j)^{\frac{p(1+\xi)}{4}})ds\right]\!\!\pp^{\frac{\xi}{1+\xi}}\left(|G|\geq
CN^{\frac{1}{2}-\alpha_1}\right).
\end{align*}
The first term of the product is bounded whereas, by the usual bound
of the tail of the normal law $\forall t>0, \pp(|G| \geq t) \leq
\frac{2e^{-\frac{t^2}{2}}}{t\sqrt{2\pi}}$, the product of the second
term by any polynomial function of $N$ is bounded.

Now for $l\in\{1,\hdots,L\}$, $\ee\left[\left|A_j-D_j\right|^{2p}\ind_{\{|Y_{t_j}|<N^{-\alpha_{l+1}}\}}\right]$ is not greater than
\begin{align*}
   \frac{1}{\delta_N^{2p}}\ee\left[\left|\int_{t_{j}}^{t_{j+1}}
(t_{j+1}-s)\left((\nu\psi'(Y_s)-\nu
\psi'(Y_{t_j}))dW_s+((b\psi'+\frac{\nu^2}{2}\psi'')(Y_s))ds\right)\right|^{2p}\ind_{\{|Y_{t_j}|<N^{-\alpha_{l+1}}\}}\right],
\end{align*}
with $b(y)=\kappa(\theta-y)$. As, by convention, $N^{-\alpha_{L+1}}=+\infty$, reasoning like in the estimation of $\overline{I}^j_2$ in the proof
of Theorem \ref{THfort}, one checks that the term with index $l=L$ in the sum in the right-hand-side of
\eqref{decompsteinstein} is smaller than $C \psi(N^{p\alpha_L})N^{-2p}$.
In the quadratic Gaussian model, since $\psi'(y)=3y^3$, using Burckholder-Davis-Gundy inequality then Hölder's inequality, one obtains that for $\xi>0$,
\begin{align*}
   \ee&\left[\left|\int_{t_{j}}^{t_{j+1}}
(t_{j+1}-s)(\psi'(Y_s)-\psi'(Y_{t_j}))dW_s\right|^{2p}\ind_{\{|Y_{t_j}|<N^{-\alpha_{l+1}}\}}\right]\\&\leq \frac{C}{N^{p-1}}\int_{t_j}^{t_{j+1}}(t_{j+1}-s)^{2p}\ee\left[|Y_s-Y_{t_j}|^{2p}(|Y_s|^{4p}+|Y_{t_j}|^{4p})\ind_{\{|Y_{t_j}|<N^{-\alpha_{l+1}}\}}\right]ds\\
&\leq \frac{C}{N^{3p-1}}\int_{t_j}^{t_{j+1}}\ee^{\frac{\xi}{1+\xi}}\left[\left|Y_s-Y_{t_j}\right|^{\frac{2p(1+\xi)}{\xi}}\right]\ee^{\frac{1}{1+\xi}}\left[(|Y_s|^{4p(1+\xi)}+|Y_{t_j}|^{4p(1+\xi)})\ind_{\{|Y_{t_j}|<N^{-\alpha_{l+1}}\}}\right]ds\\
&\leq \frac{C}{N^{3p-1}}\int_{t_j}^{t_{j+1}}\frac{C}{N^p}\times \frac{C\pp^{\frac{1}{1+\xi}}\left(|Y_{t_j}|<N^{-\alpha_{l+1}}\right)}{N^{4p\alpha_{l+1}}}ds\leq C\frac{\pp^{\frac{1}{1+\xi}}\left(|Y_{t_j}|<N^{-\alpha_{l+1}}\right)}{N^{4p(1+\alpha_{l+1})}}.
\end{align*}
In the Stein and Stein model, since $\psi'(y)=2y$, one can only take advantage of the indicator function in the probability in the numerator and the power of $N$ in the denominator is reduced to $4p$. In both models, the same bound with $\xi=0$ can be derived for $\ee\left[\left|\int_{t_{j}}^{t_{j+1}}
(t_{j+1}-s)((b\psi'+\frac{\nu^2}{2}\psi'')(Y_s))ds\right|^{2p}\ind_{\{|Y_{t_j}|<N^{-\alpha_{l+1}}\}}\right]$ and one concludes that
$$\ee\left[\left|A_j-D_j\right|^{2p}\ind_{\{|Y_{t_j}|<N^{-\alpha_{l+1}}\}}\right]\leq\begin{cases}C\frac{\pp^{\frac{1}{1+\xi}}\left(|Y_{t_j}|<N^{-\alpha_{l+1}}\right)}{N^{2p}}\mbox{ in the Stein and Stein model}\\C\frac{\pp^{\frac{1}{1+\xi}}\left(|Y_{t_j}|<N^{-\alpha_{l+1}}\right)}{N^{2p(1+2\alpha_{l+1})}}\mbox{ in the quadratic Gaussian model}\end{cases}.$$
Plugging the
three estimations together with \eqref{contprobzero} in \eqref{decompsteinstein}, one deduces that in the Stein and Stein model, when $y_0\neq 0$,
$$\ee\left[\left|\sqrt{A_j}-\sqrt{D_j}\right|^{2p}\right]\leq C\left(N^{-(2p+1)\alpha_1}+\sum_{l=1}^{L-1} N^{-2p(1-\alpha_l)-\frac{\alpha_{l+1}}{1+\xi}}+N^{-2p(1-\alpha_L)}\right).$$
Now we may suppose that $\varepsilon<\frac{1}{2}$ since the smaller
$\varepsilon$ is the stronger the statement of the proposition is.
We choose $\alpha_l=\frac{1}{2}-\frac{l\varepsilon}{2p+1}$ for
$l\in\{1,\hdots,L-1\}$ with
$L=\lceil\frac{(2p+1)(1-2\varepsilon)}{4p\varepsilon}\rceil$ and
$\alpha_L=\frac{2p-1+2\varepsilon}{4p}$ and
$\xi=\frac{2\varepsilon}{1-2\varepsilon}$. Then
$(2p+1)\alpha_1=2p(1-\alpha_L)=\frac{2p+1}{2}-\varepsilon$ and for
$l\in\{1,\hdots,L-1\}$,
$2p(1-\alpha_l)+\frac{\alpha_{l+1}}{1+\xi}\geq\frac{2p+1}{2}+(2pl-\frac{l+1}{1+\xi})\frac{\varepsilon}{2p+1}-\frac{\xi}{2(1+\xi)}\geq
\frac{2p+1}{2}+0-\varepsilon$.

In the quadratic Gaussian model, plugging the
three estimations in \eqref{decompsteinstein}, one obtains that
\begin{equation}
   \ee\left[\left|\sqrt{A_j}-\sqrt{D_j}\right|^{2p}\right]\leq C\left(N^{-(4p+\ind_{\{y_0\neq 0\}})\alpha_1}+\sum_{l=1}^{L-1} N^{-2p(1+2(\alpha_{l+1}-\alpha_l))-\ind_{\{y_0\neq 0\}}\frac{\alpha_{l+1}}{1+\xi}}+N^{-2p(1-2\alpha_L)}\right).\label{majoquadgaus}
\end{equation}
We choose $\alpha_l=\frac{2p-l\varepsilon}{4p}$ for $l\in\{1,\hdots,L-1\}$ with
$L=\lceil\frac{2p}{\varepsilon}\rceil-1$ and $\alpha_L=\frac{\varepsilon}{4p}$.
Then $4p\alpha_1=2p(1-2\alpha_L)=2p-\varepsilon=2p(1+2(\alpha_{l+1}-\alpha_l))$ for $l\in\{1,\hdots,L-2\}$ and
$2p(1+2(\alpha_{L}-\alpha_{L-1}))=\left(\lceil\frac{2p}{\varepsilon}\rceil-1\right)\varepsilon\geq 2p-\varepsilon$.
\end{proof}

\begin{remark}
\label{3demi}
\begin{itemize}
\item In the quadratic Gaussian model, when $y_0\neq 0$, choosing $\xi=1$, $\alpha_l=\frac{2p}{4p+1}\left(\frac{8p}{8p+1}\right)^{l-1}$
for $l\in\{1,\hdots,L\}$ with $L=\lceil\frac{2\log\log
N}{\log\frac{8p+1}{8p}}\rceil$ in \eqref{majoquadgaus}, one obtains
that $\ee\left[\max_{0\leq k\leq N}
\Big|\widetilde{X}_{t_k}-\widetilde{X}^N_{t_k}\Big|^{2p}\right] \leq
\frac{C\log\log N}{N^{2p}}$ .
\item The fact that we can simulate exactly the volatility process without
    affecting the order of convergence of the scheme is yet another
    advantage of our approach over the Cruzeiro \textit{et al.}
    \cite{CruzeiroMalliavinThalmaier} scheme. On the other hand, the Kahl
    and Jäckel \cite{KahlJackel} scheme allows the exact simulation of
    $(Y_t)_{t\in[0,T]}$. Applied to the SDE (\ref{SV2}), it writes as
\begin{equation}
\label{IJKscheme}
\begin{array}{l}
\ds
X^{IJK}_{t_{k+1}}=X^{IJK}_{t_{k}}+\left(r-\frac{f^2(Y_{t_{k+1}})+f^2(Y_{t_{k}})}{4}\right)\delta_N
+\rho f(Y_{t_k})\Delta
W_{k+1}\quad\quad\quad\quad\\[2mm]
\ds
\quad\quad\quad\quad\quad\quad+\sqrt{1-\rho^2}\frac{f(Y_{t_{k+1}})+f(Y_{t_{k}})}{2}
\Delta B_{k+1}+\frac{\rho \nu}{2} f'(Y_{t_k})\left(\left(\Delta
W_{k+1}\right)^2-\delta_N\right).\end{array}\end{equation} Note that it is
close to our scheme insofar as it takes advantage of the structure of the
SDE (for example, unlike the Cruzeiro \textit{et al.}
\cite{CruzeiroMalliavinThalmaier} scheme, it allows the use of the coupling
introduced in Remark \ref{MLMC}). The main difference, which explains why
our scheme has better weak trajectorial convergence order, is that we
discretize more accurately the integral of $f(Y_t)$ with respect to the
Brownian motion $(B_t)_{t\in[0,T]}$. If, instead of a trapezoidal method,
one uses the same discretization as for the WeakTraj\_1 scheme, then it can
be shown that this modified IJK scheme will exhibit a first order weak
trajectorial convergence.

\item It is possible to improve the convergence at fixed times up to the
    order $\frac{3}{2}$. Following Lapeyre and Temam \cite{LapeyreTemam}
    who approximate an integral of the form $\int_{t_k}^{t_{k+1}} g(Y_s)ds$
    for a twice differentiable function $g$ by $\delta_N g(Y_{t_k})+\nu
    g'(Y_{t_k}) \int_{t_k}^{t_{k+1}} (W_s-W_{t_k})
    ds+(\kappa(\theta-Y_{t_k})g'(Y_{t_k})+\frac{\nu^2}{2}g''(Y_{t_k}))\frac{\delta_N^2}{2}$,
    we obtain the following scheme

\vspace{2mm}

\framebox{\textbf{OU\_Improved scheme}}
\begin{equation}
\begin{array}{l} \ds
\widetilde{X}_{t_{k+1}}^N=\widetilde{X}_{t_{k}}^N+\rho
\left(F(Y_{t_{k+1}})-F(Y_{t_{k}})\right)+
\widetilde{h}_k+\sqrt{1-\rho^2}
\sqrt{\widetilde{\psi}_k}\,\,\Delta B_{k+1},\\[4mm]
\end{array}
\label{strong_Y-OU_3demi}
\end{equation}
where $\widetilde{h}_k=\delta_N h(Y_{t_k})+\nu h'(Y_{t_k})
\int_{t_k}^{t_{k+1}}(W_s-W_{t_k})
ds+(\kappa(\theta-Y_{t_k})h'(Y_{t_k})+\frac{\nu^2}{2}h''(Y_{t_k}))\frac{\delta_N^2}{2}$
and \\$\widetilde{\psi}_k=\left(\psi(Y_{t_k})+\frac{\nu
\psi'(Y_{t_k})}{\delta_N}\int_{t_k}^{t_{k+1}}
(W_s-W_{t_k})ds+(\kappa(\theta-Y_{t_k})\psi'(Y_{t_k})+\frac{\nu^2}{2}\psi''(Y_{t_k}))\frac{\delta_N}{2}\right)
\vee \underline{\psi}.$

Mimicking the proof of Theorem \ref{THfort}, one can show that
\[\max_{0\leq k\leq N}
\ee\left[\left|\widehat{X}_{t_k}-\widehat{X}_{t_{k+1}}^N\right|^2\right]=\mathcal{O}\left(N^{-3}\right),\]
where $\widehat{X}_{t_k}$ and $\widehat{X}_{t_{k+1}}^N$ have
respectively the same distribution as $X_{t_k}$ and
$\widetilde{X}^N_{t_k}$ :
\[\widehat{X}_{t_k}=X_0+\rho(F(Y_{t_k})-F(y_0))+\int_0^{t_k}
h(Y_s)ds+\sqrt{1-\rho^2}\sqrt{\frac{1}{t_k}\int_0^{t_k}
\psi(Y_s)ds}\,\,B_{t_k}\] and
\[\widehat{X}_{t_{k}}^N=X_0+\rho
\left(F(Y_{t_{k}})-F(y_0)\right)+
\sum_{j=0}^{k-1}\widetilde{h}_j+\sqrt{1-\rho^2}
\sqrt{\frac{\delta_N}{t_k}\sum_{j=0}^{k-1}\widetilde{\psi}_j}\,\,B_{t_k}.\]
\end{itemize}
\end{remark}

As for the stock, we can prove the same convergence result under
some additional assumptions which are more explicit than assumption
($\mathcal{H}$\ref{hypStock}) of Corollary \ref{CorolStock}. To do
so, let us make the following changes in our scheme so that we can
control its exponential moments :
\begin{equation}
\label{strong_modif} \begin{array}{l} \ds
\widetilde{X}_{t_{k+1}}^N=\widetilde{X}_{t_{k}}^N+\rho
\left(F(Y_{t_{k+1}})-F(Y_{t_{k}})\right)+ \delta_N h(Y_{t_k})\\[4mm]
\ds \quad \quad +\sqrt{1-\rho^2}\sqrt{\left(\psi(Y_{t_k})+\frac{\nu
\psi'(Y_{t_k})}{\delta_N}\int_{t_k}^{t_{k+1}} (W_s-W_{t_k})ds\right)
\wedge \widehat{\psi}(Y_{t_k}) \vee \underline{\psi}}\,\, \Delta
B_{k+1}.
\end{array}
\end{equation}

\begin{proposition}
Suppose that $Y$ is solution of (\ref{OU}) and that the scheme is
defined by (\ref{strong_modif}). \\Under the assumptions
($\mathcal{H}$\ref{fC3}), ($\mathcal{H}$\ref{hypTh1_YOU}) and
($\mathcal{H}$\ref{hypOU}) of Theorem \ref{THfort_Y-OU} and if
\begin{hypo}
\label{hypStock_YOU} there exists $\beta\in(0,1)$ and $K>0$ such
that $\forall y \in \rr$
\[\begin{array}{rcl}
\ds |h(y)|+ |F(y)|&\leq&\ds K(1+|y|^{1+\beta})\\[3mm]
\ds |f(y)|&\leq&\ds K(1+|y|^{\beta})\\[3mm]
|f'(y)|&\leq&K |f(y)|\end{array}\]
\end{hypo}
then, $\forall p \geq 1$, there exists a positive constant $C$
independent of $N$ such that
\[\ee\left[\max_{0\leq k\leq N}
\left|e^{\widetilde{X}_{t_k}}-e^{\widetilde{X}^N_{t_k}}\right|^{2p}\right]
\leq \frac{C}{N^{2p}}.\] The same result holds true if one replaces
assumption ($\mathcal{H}$\ref{hypOU}) by assumption
($\mathcal{H}$\ref{spos}).
\end{proposition}

\begin{proof} We go over the proof of Corollary
\ref{CorolStock}. The fact that $\ee\left[\max_{0\leq k\leq N}
\left|\widetilde{X}_{t_k}-\widetilde{X}^N_{t_k}\right|^{4p}\right]=\mathcal{O}(\frac{1}{N^{4p}})$
is not a straightforward consequence of Theorem \ref{THfort_Y-OU}
anymore because we have introduced some changes in our scheme.
However, looking through the proof of the theorem, one can see that
it is enough to prove the following inequality : $\forall j \in
\{0,\dots, N-1\}$
\begin{equation}
\label{cutoff} \ee\left[\left|\sqrt{\frac{1}{\delta_{N}}
\int_{t_{j}}^{t_{j+1}}
   \psi(Y_{s})ds}-\sqrt{\left(\psi(Y_{t_j})+\frac{\nu
\psi'(Y_{t_j})}{\delta_{N}}\int_{t_j}^{t_{j+1}}
(W_s-W_{t_j})ds\right)\wedge \widehat{\psi}(Y_{t_j})\vee
\underline{\psi}}\right|^{2p}\right] \leq
\frac{C}{N^{2p}}.\end{equation} When $\overline{\psi}$ is finite,
since $\frac{1}{\delta_N}\int_{t_j}^{t_{j+1}} \psi(Y_s)ds$ is
smaller than $\widehat{\psi}(Y_{t_j})=\overline{\psi}$, the
expectation is not greater than the similar one without the new
cut-off and \eqref{cutoff} holds by the proof of Theorem
\ref{THfort_Y-OU}. When $\overline{\psi}=+\infty$, for $\epsilon>0$,
the expectation of interest is smaller than \begin{align*}
  &\ee\left[\left|\sqrt{\frac{1}{\delta_{N}}
\int_{t_{j}}^{t_{j+1}}
   \psi(Y_{s})ds}-\sqrt{\left(\psi(Y_{t_j})+\frac{\nu
\psi'(Y_{t_j})}{\delta_{N}}\int_{t_j}^{t_{j+1}}
(W_s-W_{t_j})ds\right)\vee
\underline{\psi}}\right|^{2p}\right]\\
&+\ee^{\frac{1}{1+\epsilon}}\left[\left(\frac{1}{\delta_{N}}
\int_{t_{j}}^{t_{j+1}}
   \psi(Y_{s})^pds+(\frac{3}{2}\psi(Y_{t_j}))^p\right)^{1+\epsilon}\right]\pp^{\frac{\epsilon}{1+\epsilon}}\left[\psi(Y_{t_j})+\frac{\nu
\psi'(Y_{t_j})}{\delta_{N}}\int_{t_j}^{t_{j+1}}
(W_s-W_{t_j})ds\geq\frac{3}{2}\psi(Y_{t_j})\right].
\end{align*}
We estimate the first term like in the proof of Theorem
\ref{THfort_Y-OU}. The expectation in the second term is bounded
uniformly in $N$. By hypothesis ($\mathcal{H}$\ref{hypStock_YOU}),
$\exists K<+\infty,\;\forall y\in\rr,\;|\psi'(y)|\leq K\psi(y)$ and,
for $G$ normally distributed, the probability in this second term is
smaller than $\pp\left(|G|\geq \frac{C}{\sqrt{\delta_N}}\right) $
which decreases quicker than polynomially to $0$ as $N\to\infty$.
Therefore \eqref{cutoff} holds.

\vspace{1mm}

Now, what is left to prove is that assumption
($\mathcal{H}$\ref{hypStock}) is satisfied. On the one hand, we have
that
\[\begin{array}{rcl}
\ds \ee\left[\max_{0\leq k\leq N} S_{t_k}^{4 p}\right] &=&\ds
\ee\left[\max_{0\leq k\leq N} \left(S_0+\int_0^{t_k} rS_sds+\int_0^{t_k} f(Y_s)S_s \left(\rho dW_s+\sqrt{1-\rho^2}dB_s\right)\right)^{4p}\right]\\[5mm]
&\leq&\ds C \left(1+\int_0^T\ee\left(S_t^{4p}(1+f^{4p}(Y_t))\right)dt\right)\\[5mm]
&\leq&\ds C \left(1+\int_0^T\sqrt{\ee(S_t^{8p})} \sqrt{\ee\left((1+f^{4p}(Y_t))^2\right)}dt\right).\\[5mm]
\end{array}\]
Thanks to assumption ($\mathcal{H}$\ref{hypStock_YOU}) and Lemma
\ref{MomentExpo}, there exists $C>0$ such that
$\sqrt{\ee\left((1+f^{4p}(Y_t))^2\right)} \leq C$. Observe that
conditionally on $(Y_t)_{t\in[0,T]}$,
\begin{equation}
X_t \sim \mathcal{N}\left(\log(s_0)+\rho (F(Y_t)-F(y_0))+\int_0^t
h(Y_s)ds\,,\,(1-\rho^2) \int_0^t f^2(Y_s)ds\right),
\end{equation}
so, by Jensen's inequality and assumption
($\mathcal{H}$\ref{hypStock_YOU})
\[\begin{array}{rcl}
\ds \ee\left(S_t^{8p}\right)&=&\ds
\ee\left(e^{8p\left(\log(s_0)+\rho (F(Y_t)-F(y_0))+\int_0^t
h(Y_s)ds\right)} e^{32p^2 (1-\rho^2) \int_0^t f^2(Y_s)ds}\right)\\[3mm]
&\leq&\ds \ee\left(e^{8p\left(\log(s_0)+\rho
(F(Y_t)-F(y_0))\right)} \frac{1}{t}\int_0^te^{t\left(8p h(Y_s)+32p^2(1-\rho^2)f^2(Y_s)\right)}ds\right)\\[3mm]
&\leq&\ds C \ee\left(e^{C \sup_{0\leq t\leq
T}|Y_t|^{1+\beta}}\right).
\end{array}\] Using Lemma \ref{MomentExpo}, we deduce that
$\ee\left[\max_{0\leq k\leq N} S_{t_k}^{4 p}\right] <\infty$.

On the other hand, using Cauchy-Schwarz inequality, we have that
\[\begin{array}{rcl}
\ds \ee\left[\max_{0\leq k\leq N} e^{4
p\widetilde{X}^N_{t_k}}\right] &=&\ds \ee\left[\max_{0\leq k\leq N}
\exp\left(4 p
\left(X_0+\rho(F(Y_{t_k})-F(y_0))+\sum_{j=0}^{k-1}\delta_N
h(Y_{t_j})+\sum_{j=0}^{k-1}\sqrt{1-\rho^2}\right.\right.\right.\\[5mm]
&&\ds \left.\left.\left.\times\sqrt{\left(\psi(Y_{t_j})+\frac{\nu
\psi'(Y_{t_j})}{\delta_N}\int_{t_j}^{t_{j+1}}(W_s-W_{t_j})ds\right)\wedge\widehat{\psi}(Y_{t_j})\vee\underline{\psi}}
\, \Delta B_{j+1}\right)\right)\right]
\\[5mm]
&\leq& \sqrt{\widetilde{E}^N_1}\sqrt{\widetilde{E}^N_2},
\end{array}\]
where \[\widetilde{E}^N_1=\ee\left[\max_{0\leq k\leq
N}e^{8p\left(X_0+\rho(F(Y_{t_k})-F(y_0))+\sum_{j=0}^{k-1}\delta_N
h(Y_{t_j})\right)}\right]
\]and\[\widetilde{E}^N_2=\ee\left[\max_{0\leq k\leq
N}e^{8p\sqrt{1-\rho^2}\sum_{j=0}^{k-1}\sqrt{\left(\psi(Y_{t_j})+\frac{\nu
\psi'(Y_{t_j})}{\delta_N}\int_{t_j}^{t_{j+1}}(W_s-W_{t_j})ds\right)\wedge\widehat{\psi}(Y_{t_j})\vee\underline{\psi}}
\, \Delta B_{j+1}}\right].\] Using the same argument as before, we
show that $\widetilde{E}^N_1 \leq C \ee\left(e^{C \sup_{0\leq t\leq
T}|Y_t|^{1+\beta}}\right)<\infty.$

Denote by $D_j=\left(\psi(Y_{t_j})+\frac{\sigma
\psi'(Y_{t_j})}{\delta_{N}}\int_{t_j}^{t_{j+1}} (W_s-W_{t_j})ds\right)\wedge
\widehat{\psi}(Y_{t_j})\vee \underline{\psi}$. Using Doob's maximal inequality
for the positive submartingale $\left(e^{4p\sqrt{1-\rho^2} \sum_{j=0}^{k-1}
\sqrt{D_j} \Delta B_{j+1}}\right)_{0\leq k \leq N}$ (see Theorem 3.8 p. 13 of
Karatzas and Shreve \cite{KaratzasShreve} for example), we also have that
\[\begin{array}{rcl}
\ds \widetilde{E}^N_2&\leq&\ds 4 \ee\left(
e^{8p\sqrt{1-\rho^2}\sum_{j=0}^{N-1}
\sqrt{D_j} \, \Delta B_{j+1}}\right)\\[5mm]
&=&\ds 4 \ee\left(\prod_{j=0}^{N-1}
e^{32p^2\delta_N(1-\rho^2)D_j}\right)\\[5mm]
&\leq&\ds 4 \ee\left(\max_{0\leq k\leq N-1}e^{32p^2(1-\rho^2)\widehat{\psi}(Y_{t_j})}\right).\\[5mm]
\end{array}
\]
By virtue of assumption ($\mathcal{H}$\ref{hypStock_YOU}),
$\widetilde{E}^N_2 < \infty$ which concludes the proof.
\end{proof}

\section{A second order weak scheme}

Integrating the first stochastic differential equation in
(\ref{SV3}) gives
\begin{equation}
\label{X} X_t=\log(s_0)+\rho (F(Y_t)-F(y_0))+\int_0^t
h(Y_s)ds+\sqrt{1-\rho^2} \int_0^t f(Y_s) dB_s.
\end{equation}

We are only left with an integral with respect to time which can be
handled by the use of a trapezoidal scheme and a stochastic integral
where the integrand is independent of the Brownian motion. Hence,
conditionally on $(Y_t)_{t\in[0,T]}$,
\begin{equation}
X_T \sim \mathcal{N}\left(\log(s_0)+\rho
(F(Y_T)-F(y_0))+m_T\,,\,(1-\rho^2) v_T\right),
\end{equation}
where $m_T=\int_0^T h(Y_s)ds$ and $v_T=\int_0^T f^2(Y_s)ds$. This
suggests that, in order to properly approximate the law of $X_T$,
one should accurately approximate the law of $Y_T$ and carefully
handle integrals with respect to time of functions of the process
$(Y_t)_{t\in[0,T]}$. We thus define our weak scheme as follows

\vspace{4mm}

\framebox{\textbf{Weak\_2 scheme}}
\begin{equation}
\label{Weak} \overline{X}^N_T=\log(s_0)+\rho
(F(\overline{Y}^N_{T})-F(y_0))+\overline{m}^N_T+\sqrt{(1-\rho^2)
\overline{v}^N_T} G
\end{equation}
where
$\overline{m}^N_T=\delta_N\sum_{k=0}^{N-1}\frac{h(\overline{Y}^N_{t_k})+h(\overline{Y}^N_{t_{k+1}})}{2}$,
$\overline{v}_T^N=\delta_N\sum_{k=0}^{N-1}\frac{f^2(\overline{Y}^N_{t_k})+f^2(\overline{Y}^N_{t_{k+1}})}{2}$,
$(\overline{Y}^N_{t_k})_{0\leq k\leq N}$ is the Ninomiya-Victoir
scheme of $(Y_t)_{t\in[0,T]}$ and $G$ is an independent centered
reduced Gaussian random variable. Note that, conditionally on
$(\overline{Y}^N_{t_k})_{0\leq k\leq N}, \overline{X}^N_t$ is also a
Gaussian random variable with mean $\log(s_0)+\rho
(F(\overline{Y}^N_{T})-F(y_0))+\overline{m}^N_T$ and variance
$(1-\rho^2)\overline{v}^N_T$.

\vspace{2mm}

It is well known that the Ninomiya and Victoir \cite{NinomiyaVictoir} scheme is
of weak order two. For the sake of completeness, we give its definition in our
setting :
\[\left\{\begin{array}{l}
\overline{Y}^N_0=y_0\\
\forall 0\leq k\leq N-1,
\overline{Y}^N_{t_{k+1}}=\exp\left(\frac{T}{2N}V_0\right)\exp\left((W_{t_{k+1}}-W_{t_k})V\right)\exp\left(\frac{T}{2N}V_0\right)(\overline{Y}^N_{t_k}),
\end{array}\right.\]
where $V_0: x\mapsto b(x)-\frac{1}{2}\sigma \sigma'(x)$ and $V: x\mapsto
\sigma(x)$. The notation $\exp(tV)(x)$ stands for the solution, at time $t$ and
starting from $x$, of the ODE $\eta'(t)=V(\eta(t))$. What is nice with our
setting is that we are in dimension one and thus such ODEs can be solved
explicitly. Indeed, if  $\zeta$ is a primitive of $\frac{1}{V}$ :
$\zeta(t)=\int_0^t\frac{1}{V(s)}ds$, then the solution writes as
$\eta(t)=\zeta^{-1}\left(t+\zeta(x)\right)$.\\ Note that our scheme can be seen
as a splitting scheme for the SDE satisfied by $(Z_t=X_t-\rho F(Y_t),Y_t)$ :
\begin{equation}
\label{SV4} \left\{\begin{array}{rcl} dZ_t&=&h(Y_t)dt+\sqrt{1-\rho^2}f(Y_t)dB_t\\[1mm]
dY_t&=&b(Y_t)dt+\sigma(Y_t)dW_t.
\end{array}\right.
\end{equation}
The differential operator associated to (\ref{SV4}) writes as
\[\mathcal{L}v(z,y)=h(y)\frac{\partial v}{\partial z}+b(y)\frac{\partial v}{\partial y}+\frac{\sigma^2(y)}{2}
\frac{\partial^2 v}{\partial y^2}+\frac{(1-\rho^2)}{2}f^2(y)\frac{\partial^2
v}{\partial z^2}=\mathcal{L}_Yv(z,y)+\mathcal{L}_Zv(z,y),\]where $\mathcal{L}_Y
v(z,y)=b(y)\frac{\partial v}{\partial y}+\frac{\sigma^2(y)}{2} \frac{\partial^2
v}{\partial y^2}$ and $\mathcal{L}_Zv(z,y)=h(y)\frac{\partial v}{\partial
z}+\frac{(1-\rho^2)}{2}f^2(y)\frac{\partial^2 v}{\partial z^2}$. One can check
that our scheme amounts to first integrate exactly $\mathcal{L}_Z$ over a half
time step then apply the Ninomiya-Victoir scheme to $\mathcal{L}_Y$ over a time
step and finally integrate exactly $\mathcal{L}_Z$ over a half time step.
According to results on splitting (see Alfonsi \cite{Alfonsi2} or Tanaka and
Kohatsu-Higa \cite{TanakaKohatsuHiga} for example) one expects this scheme to
exhibit second order weak convergence. Actually, according to Theorem 1.17 in
Alfonsi \cite{Alfonsi2}, our scheme has potential second order of weak
convergence. To deduce formally the order two of weak convergence, one only
needs to check regularity of the solution of the backward Kolmogorov equation
associated with the model.

\begin{remark}
\label{conditioning}
\begin{itemize}

\item As for plain vanilla options pricing, observe that, by the Romano and
    Touzi \cite{RomanoTouzi} formula,
\[\ee\left(e^{-rT} \alpha(S_T) |
(Y_t)_{t\in[0,T]}\right)=BS_{\alpha,T}\left(s_0e^{\rho(F(Y_T)-F(y_0))+m_T+(\frac{(1-\rho^2)v_T}{2T}-r)T},\frac{(1-\rho^2)v_T}{T}\right),\]
where $BS_{\alpha,T}(s,v)$ stands for the price of a European option with
pay-off $\alpha$ and  maturity $T$ in the \BS~model with initial stock
price $s$, volatility $\sqrt{v}$ and constant interest rate $r$. When, like
for a call or a put option, $BS_{\alpha,T}$ is available in a closed form,
one should approximate $\ee\left(e^{-rT} \alpha(S_T)\right)$ by
\[\frac{1}{M} \sum_{i=1}^M
BS_{\alpha,T}\left(s_0e^{\rho(F(\overline{Y}^{N,i}_T)-F(y_0))+\overline{m}^{N,i}_T+(\frac{(1-\rho^2)\overline{v}^{N,i}_T}{2T}-r)T},
\frac{(1-\rho^2)\overline{v}^{N,i}_T}{T}\right),\] where $M$ is the total
number of Monte Carlo samples and the index $i$ refers to independent
draws.

Indeed, the conditioning provides a variance reduction.

\item In the special case of an Ornstein-Uhlenbeck process driving the volatility (i.e $(Y_t)_{t\in[0,T]}$ is solution of
the SDE (\ref{OU})), one should replace the Ninomiya-Victoir scheme
by the true solution. The order two of weak convergence should then
be preserved. Moreover, one can check that the OU\_Improved scheme
(\ref{strong_Y-OU_3demi}) has also potential second order of weak
convergence. Better still, it achieves a weak trajectorial
convergence of order $\frac{3}{2}$ on the triplet
$(Y_t,m_t,v_t)_{t\in[0,T]}$ which allows for a significant
improvement of the multilevel Monte Carlo method, as we shall check
numerically.
\end{itemize}
\end{remark}

\section{Numerical comparative analysis of the proposed schemes with standard discretization methods}
We focus on the case where $(Y_t)_{t\in[0,T]}$ is an
Ornstein-Uhlenbeck process since, as mentioned in the introduction,
it encompasses several standard stochastic volatility models.

We are going to compare our schemes (WeakTraj\_1, Weak\_2 and OU\_Improved) to
the Euler scheme with exact simulation of the process $(Y_t)_{t\in[0,T]}$
driving the volatility (hereafter denoted Euler), the Kahl and Jäckel
\cite{KahlJackel} scheme (IJK) and the Cruzeiro \textit{et al.}
\cite{CruzeiroMalliavinThalmaier} scheme (CMT).

For the following numerical computations, unless otherwise stated, we are going
to consider Scott's model (\ref{Scott}). We use the same set of parameters as
in Kahl and Jäckel \cite{KahlJackel} : $S_0=100, r=0.05, T=1, \sigma_0=0.25,
y_0=0, \kappa=1,\theta=0,\nu=\frac{7\sqrt{2}}{20}, \rho=-0.2$ and $f:y\mapsto
\sigma_0e^y$.

\subsection{Theoretical computational cost per timestep}

The following table gives the computational cost per timestep for
each scheme in terms of function evaluations and random samples
needed. Except for the CMT scheme, we consider exact simulation of
$(Y_t)_{t\in[0,T]}$ which requires the simulation of one Gaussian
variable at each timestep. For the Weak2 scheme, no other simulation
per timestep is required since we only need to simulate one Gaussian
variable at the terminal time (see equation (\ref{Weak}). The IJK
scheme, as we can see according to (\ref{IJKscheme}), requires the
simulation of the two Brownian increments.

Certainly, our schemes require more computational effort per
timestep but we will see hereafter that their higher order of
convergence suffices to have better efficiency.

\begin{table}[h]
\begin{center}
\begin{tabular}{|c|c|c|}
\hline Schemes &Function evaluations&Random number samples\\\hline
WeakTraj\_1&3 ($f, f'$ and $F$)&3 Gaussian simulations \\
\hline Weak\_2 &2 ($f$ and $f'$)& 1 Gaussian simulation \\
\hline OU\_Improved &5 ($f, f', f'', f'''$ and $F$)&3 Gaussian simulations\\
\hline IJK &2 ($f$ and $f'$)&3 Gaussian simulations\\
\hline CMT &2 ($f$ and $f'$)&2 Gaussian simulations\\
\hline Euler &1 ($f$)&2 Gaussian simulations\\
\hline
\end{tabular}
\end{center}
\vspace{-6mm}
\caption{\label{tab:CompEffort} Computation effort per
timestep}
\end{table}

\subsection{Numerical illustration of strong convergence properties}
In order to illustrate the strong convergence rate of a
discretization scheme $\widehat{X}^N$, we consider the squared
$L^2$-norm of the supremum of the difference between the scheme with
time step $\frac{T}{N}$ and the one with time step $\frac{T}{2N}$ :
\begin{equation}
\label{expreStrong} \ee\left[\max_{0\leq k \leq N}
\left|\widehat{X}^N_{t_k}-\widehat{X}^{2N}_{t_{k}}\right|^2\right].
\end{equation}

This quantity will exhibit the same asymptotic behavior with respect to $N$ as
the squared $L^2$-norm of the difference between the scheme with time step
$\frac{T}{N}$ and the limiting process towards which it converges (see Alfonsi
\cite{Alfonsi}).

In Figure \ref{fig:strongLogAsset}, we draw the logarithm of the
Monte Carlo estimation of (\ref{expreStrong}) as a function of the
logarithm of the number of time steps. The number of discretization
steps is a power of $2$ varying from $2$ to $256$ and the number of
Monte Carlo samples used is equal to $M=10 \,000$. We also consider
the strong convergence of the schemes on the asset itself (see
Figure \ref{fig:strongAsset}) by computing $ \ee\left[\max_{0\leq k
\leq N}
\left|e^{\widehat{X}^N_{t_k}}-e^{\widehat{X}^{2N}_{t_{k}}}\right|^2\right]
$.

The confidence intervals of the estimations are reported in error bars in the
figures : as one can see, the number of simulations considered suffices to have
precise results. The average width of the confidence intervals reported in
figures \ref{fig:strongLogAsset} and \ref{fig:strongAsset} is equal to 0.07.
Note that, since the width of the confidence interval in the estimation of
(\ref{expreStrong}) is proportional to the standard error which should
theoretically be proportional to $N$ too, then the width of the confidence
interval expressed in logarithmic scale should be constant. We can see that
this is indeed the case, especially when the number of time-steps is large
enough.

The slopes of the regression lines are reported in
Table~\ref{tab:slopes}. For completeness sake, we give the standard
deviation of the residuals in the regression. We see that, both for
the logarithm of the asset and for the asset itself, all the schemes
exhibit a strong convergence of order $\frac{1}{2}$. Our schemes
only have a better constant.

\begin{table}[h]
\begin{center}
\begin{tabular}{|c|c|c|c|c|c|c|}
\hline &WeakTraj\_1&Weak\_2&OU\_Improved&IJK&CMT&Euler\\\hline
Log-asset&-1.01 (0.06) &-0.88 (0.03)& -0.94 (0.04)&-0.92 (0.07)&-0.98 (0.02)&-0.84 (0.08)\\
\hline Asset&-1.01 (0.06) & -0.91 (0.05)&-0.95 (0.02)&-0.88 (0.08)&-0.95 (0.06)&-0.85 (0.09)\\
\hline
\end{tabular}
\end{center}

\vspace{-6mm}
\caption{\label{tab:slopes} Slopes of the regression lines
(Strong convergence)}
\end{table}
\vspace{-2mm}
\subsubsection{Weak trajectorial convergence}
Nevertheless, as explained in Remark \ref{MLMC}, for the scheme with time
step $\frac{1}{N}$, one can replace the increments of the Brownian motion
$(B_t)_{t\in[0,T]}$ by a sequence of Gaussian random variables smartly
constructed from the scheme with time step $\frac{1}{2N}$. This particular
coupling is possible whenever the independence structure between
$(B_t)_{t\in[0,T]}$ and $(Y_t)_{t\in[0,T]}$ is preserved by the
discretization of the latter process, which is the case for all the
schemes but the CMT scheme. So we carry out this coupling and we repeat
the preceding numerical experiment. The results are put together in
Figures \ref{fig:strongLogAsset_Coupling} and
\ref{fig:strongAsset_Coupling} and in Table \ref{tab:slopes_Coupling}. The
average width of the confidence intervals is equal to 0.09.
\begin{figure}[h]
  \begin{minipage}[c]{.46\linewidth}
\hspace{-6mm}\includegraphics[scale=0.38]{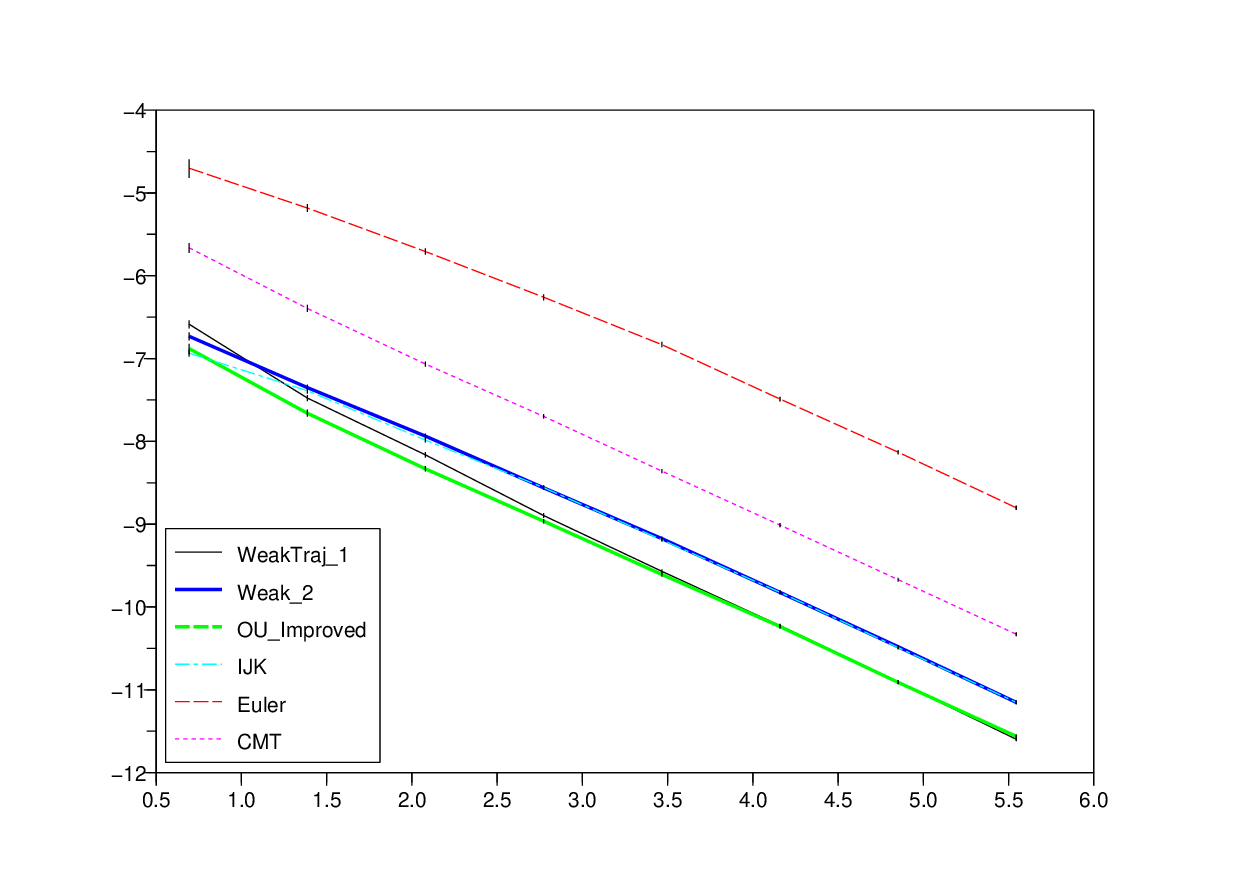}

\vspace{-6mm}
\caption{Strong convergence on the log-asset}
     \label{fig:strongLogAsset}
  \end{minipage} \hfill
  \begin{minipage}[c]{.46\linewidth}
\hspace{-3mm}\includegraphics[scale=0.38]{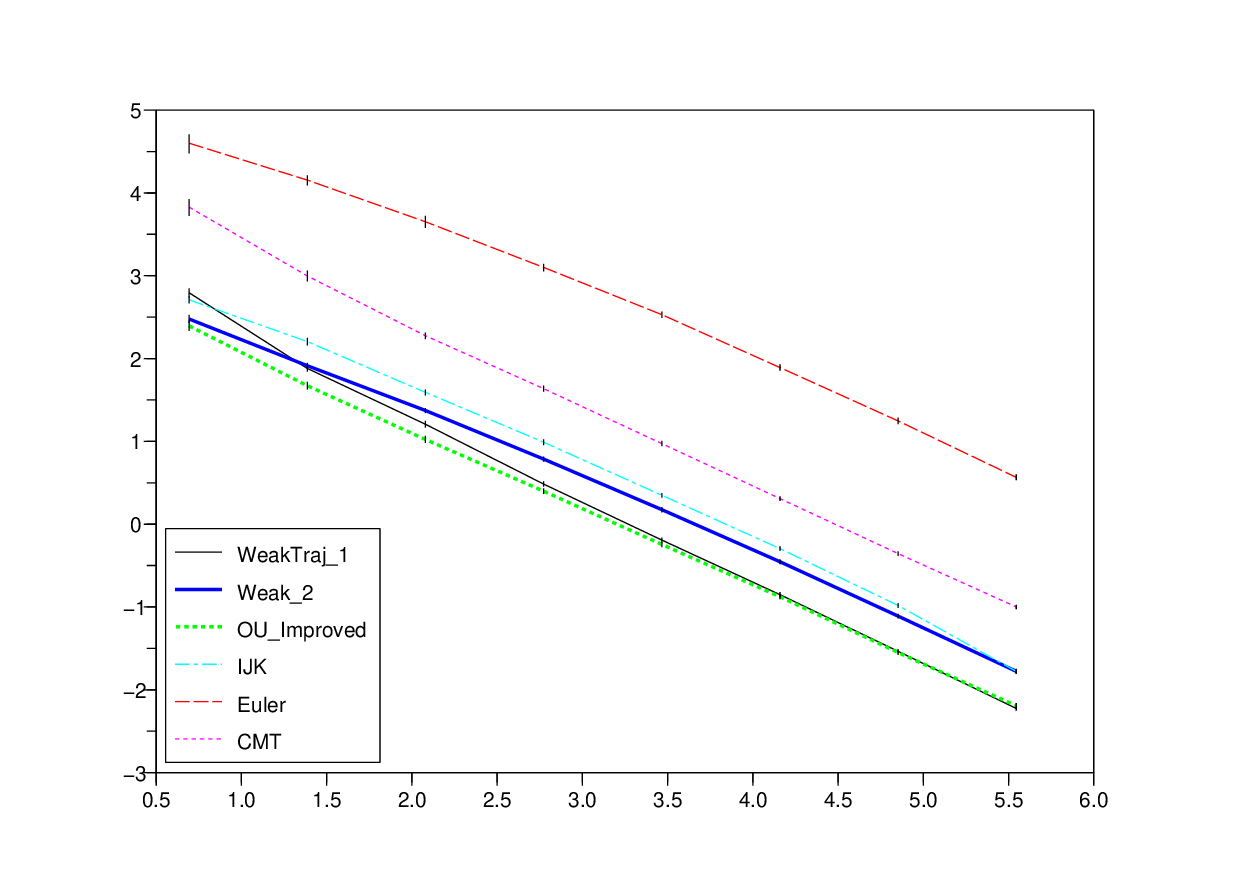}

\vspace{-6mm}
\caption{Strong convergence on the asset}
       \label{fig:strongAsset}
  \end{minipage}
\end{figure}

As expected, we see that the WeakTraj\_1 and the OU\_Improved schemes
exhibit a first order convergence rate whereas the other schemes exhibit a
$\frac{1}{2}$ order convergence rate. Note that the CMT scheme has a weak
trajectorial convergence of order one but it is much more difficult to
implement the coupling for which the convergence order is indeed equal to
one.\vspace{-2mm}

\begin{figure}[h]
  \begin{minipage}[c]{.46\linewidth}
\hspace{-6mm}\includegraphics[scale=0.38]{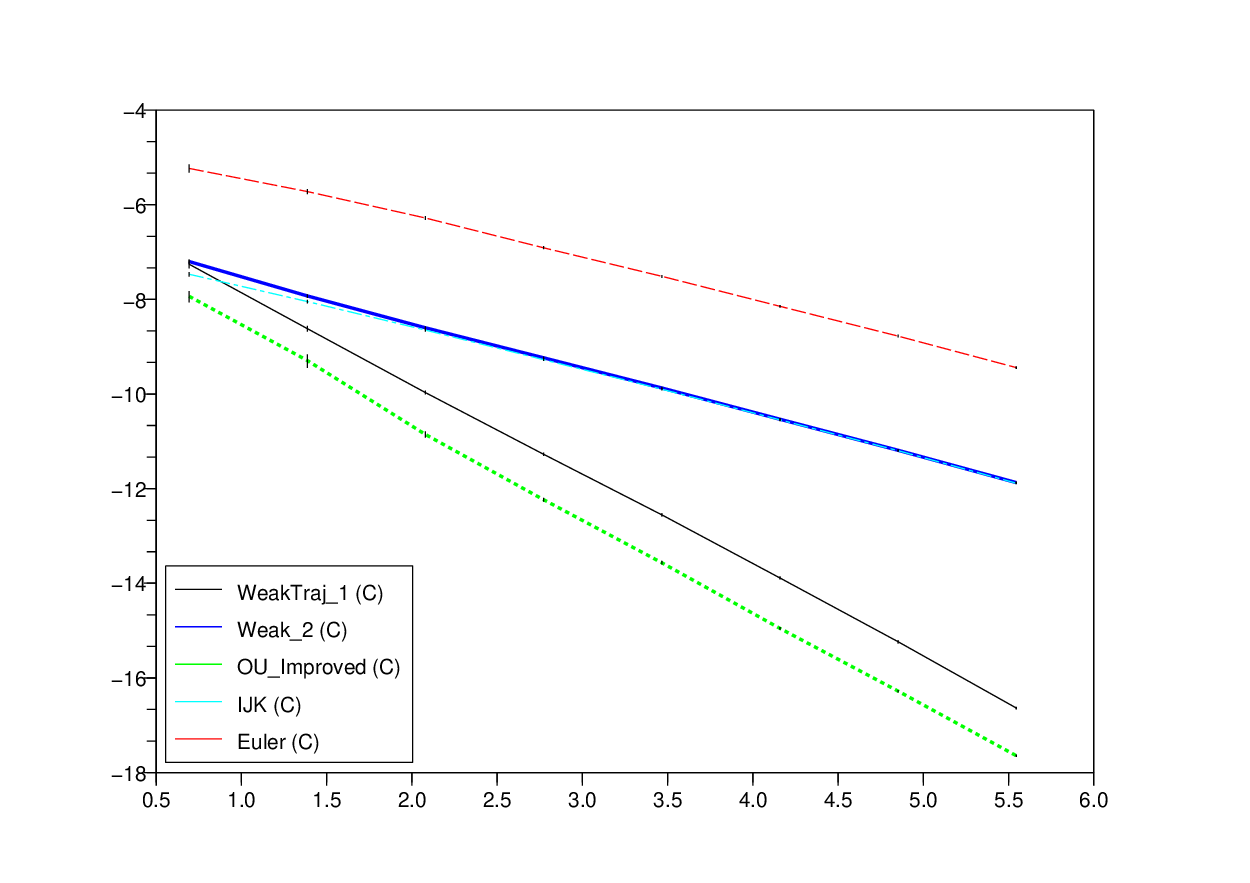}

\vspace{-6mm}
\caption{Weak trajectorial convergence on the log-asset
(with coupling)}
     \label{fig:strongLogAsset_Coupling}
  \end{minipage} \hfill
  \begin{minipage}[c]{.46\linewidth}
\hspace{-3mm}\includegraphics[scale=0.38]{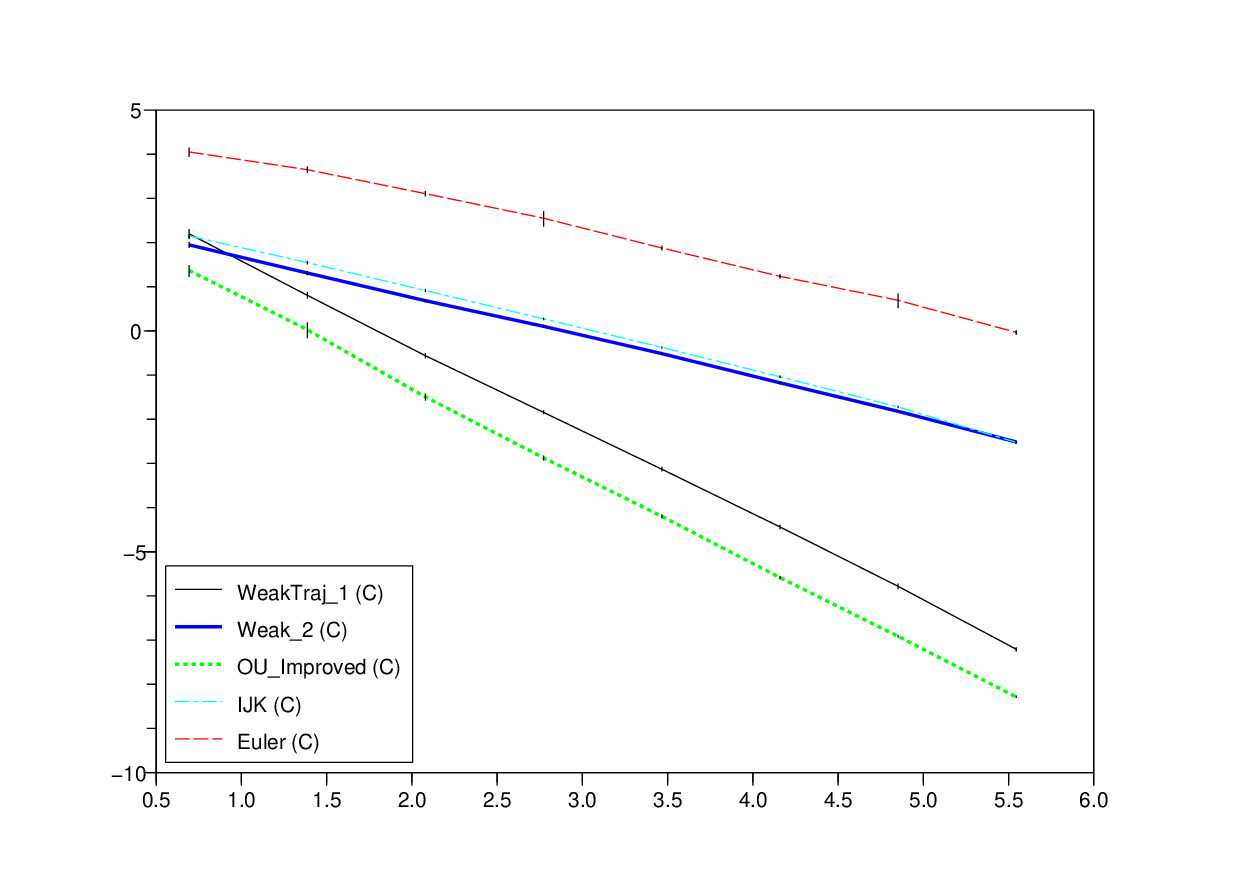}

\vspace{-6mm}
\caption{Weak trajectorial convergence on the asset (with
coupling)}
       \label{fig:strongAsset_Coupling}
  \end{minipage}
\end{figure}
\begin{table}[h]
\begin{center}
\begin{tabular}{|c|c|c|c|c|c|c|}
\hline &WeakTraj\_1&Weak\_2&OU\_Improved&IJK&CMT&Euler\\\hline
Log-asset&-1.92 (0.03) & -0.91 (0.02)& -1.99 (0.06)& -0.95 (0.03)& -- & -0.85 (0.05)\\
\hline
Asset& -1.92 (0.04)&-0.95 (0.03)& -2 (0.05)& -0.91 (0.06)& -- &  -0.87 (0.09)\\
\hline
\end{tabular}
\end{center}
\vspace{-6mm}
\caption{\label{tab:slopes_Coupling} Slopes of the
regression lines (Weak trajectorial convergence)}
\end{table}

We repeat the same numerical experiments for the Stein \& Stein and the
quadratic Gaussian model. The results are reported in figures
\ref{fig:QGstrong} and \ref{fig:SSstrong} and in Table \ref{tab:slopes_Coupling_SSQG}. We observe that
the theoretical results stated in Proposition (\ref{SSqG}) are confirmed
by the numerical findings : the slope of the regression line is
approximately equal to $1.3$ for both Weak\_Traj1 and OU\_Improved schemes
in the Stein\&Stein model whereas for the quadratic Gaussian model, it is
approximately equal to 2. \vspace{-2mm}
\begin{figure}[h]
  \begin{minipage}[c]{.46\linewidth}
\hspace{-6mm}\includegraphics[scale=0.4]{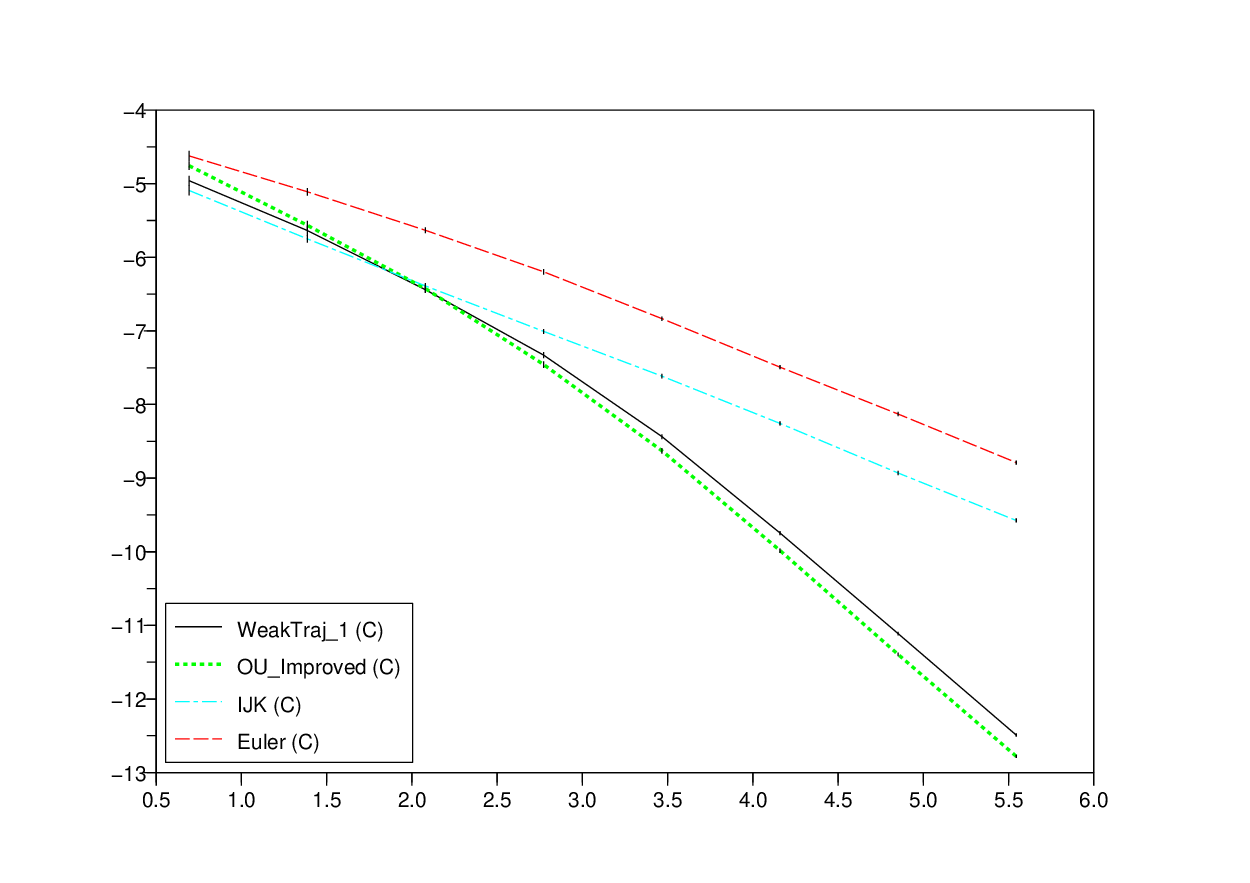}

\vspace{-6mm}
\caption{Quadratic Gaussian model - Weak trajectorial
convergence on the asset (with coupling)}
     \label{fig:QGstrong}
  \end{minipage} \hfill
  \begin{minipage}[c]{.46\linewidth}
\hspace{-3mm}\includegraphics[scale=0.4]{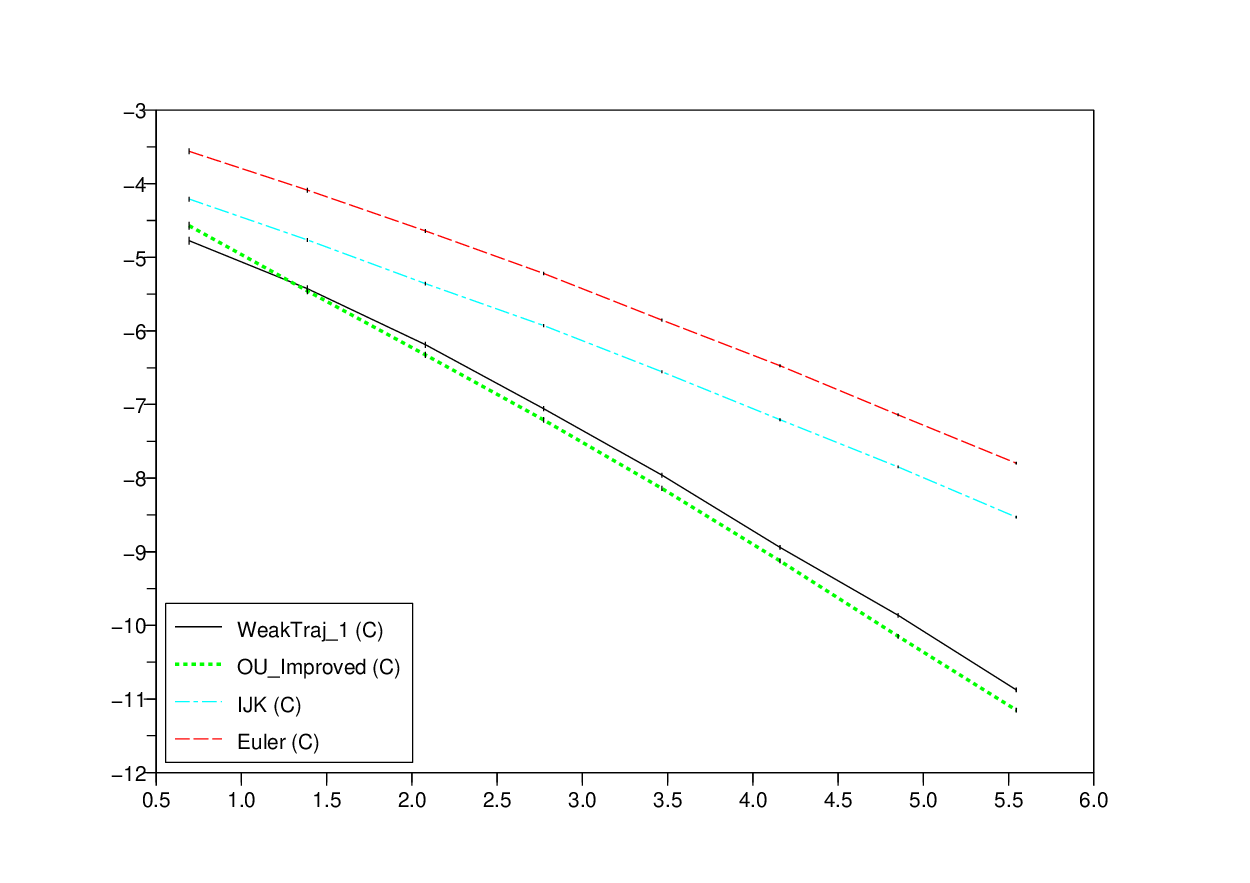}

\vspace{-6mm}
\caption{Stein \& Stein model - Weak trajectorial
convergence on the asset (with coupling)}
       \label{fig:SSstrong}
  \end{minipage}
\end{figure}

\begin{table}[h]
\begin{center}
\begin{tabular}{|c|c|c|c|c|}
\hline &WeakTraj\_1&OU\_Improved&IJK&Euler\\
\hline
Quadratic Gaussian model - Asset& -1.95 (0.02)&-1.99 (0.01)& -0.94 (0.07)& -0.94 (0.01)\\
\hline
Stein\&Stein model - Asset& -1.3 (0.12)&-1.35 (0.07)& -0.89 (0.04)& -0.87 (0.06)\\
\hline
\end{tabular}
\end{center}
\vspace{-6mm}
\caption{\label{tab:slopes_Coupling_SSQG} Slopes of the
regression lines (Weak trajectorial convergence) - Quadratic Gaussian and
Stein\&Stein models}
\end{table}

\subsubsection{Convergence at terminal time}
We consider now convergence at terminal time, precisely the squared
$L^2$-norm of the difference between the terminal values of the
schemes with time steps $\frac{T}{N}$ and $\frac{T}{2N}$ :
\begin{equation}
\label{expreStrongT}
\ee\left[\left|\widehat{X}^N_T-\widehat{X}^{2N}_T\right|^2\right].
\end{equation}
Note that we introduce a coupling : we write the schemes straight at
the terminal time as we did for the Weak\_2 scheme (see
(\ref{Weak})) and we generate the terminal values of the schemes
with time steps $\frac{T}{N}$ and $\frac{T}{2N}$ using the same
single normal random variable to simulate the stochastic integral
w.r.t. $(B_t)_{t\in [0,T]}$. Once again, it is possible to proceed
alike for all the schemes but the CMT scheme. For the latter, we
simulate the scheme at all the intermediate discretization times to
obtain the value at terminal time.

We also consider the convergence at terminal time of the asset
itself. We report the numerical results in Figures
\ref{fig:strongTLogAsset} and \ref{fig:strongTAsset} and give the
slopes of the regression lines in Table \ref{tab:slopes_T}.

\begin{table}[!h]
\begin{center}
\begin{tabular}{|c|c|c|c|c|c|c|}
\hline &WeakTraj\_1&Weak\_2&OU\_Improved&IJK&CMT&Euler\\\hline
Log-asset& -2.03 (0.04)& -2 (0.05)& -2.97 (0.03)& -1.97 (0.02)& -1.05 (0.04)& -1.34 (0.19)\\
\hline
Asset& -2.02 (0.04)& -1.98 (0.04)&-2.97 (0.06)&-1.95 (0.03) & -1.08 (0.08)&-1.34 (0.18)\\
\hline
\end{tabular}
\end{center}
\vspace{-6mm}
\caption{\label{tab:slopes_T} Slopes of the regression lines
(Convergence at terminal time)}
\end{table}
\vspace{-2mm}

\begin{figure}
  \begin{minipage}[c]{.46\linewidth}
\hspace{-6mm}\includegraphics[scale=0.4]{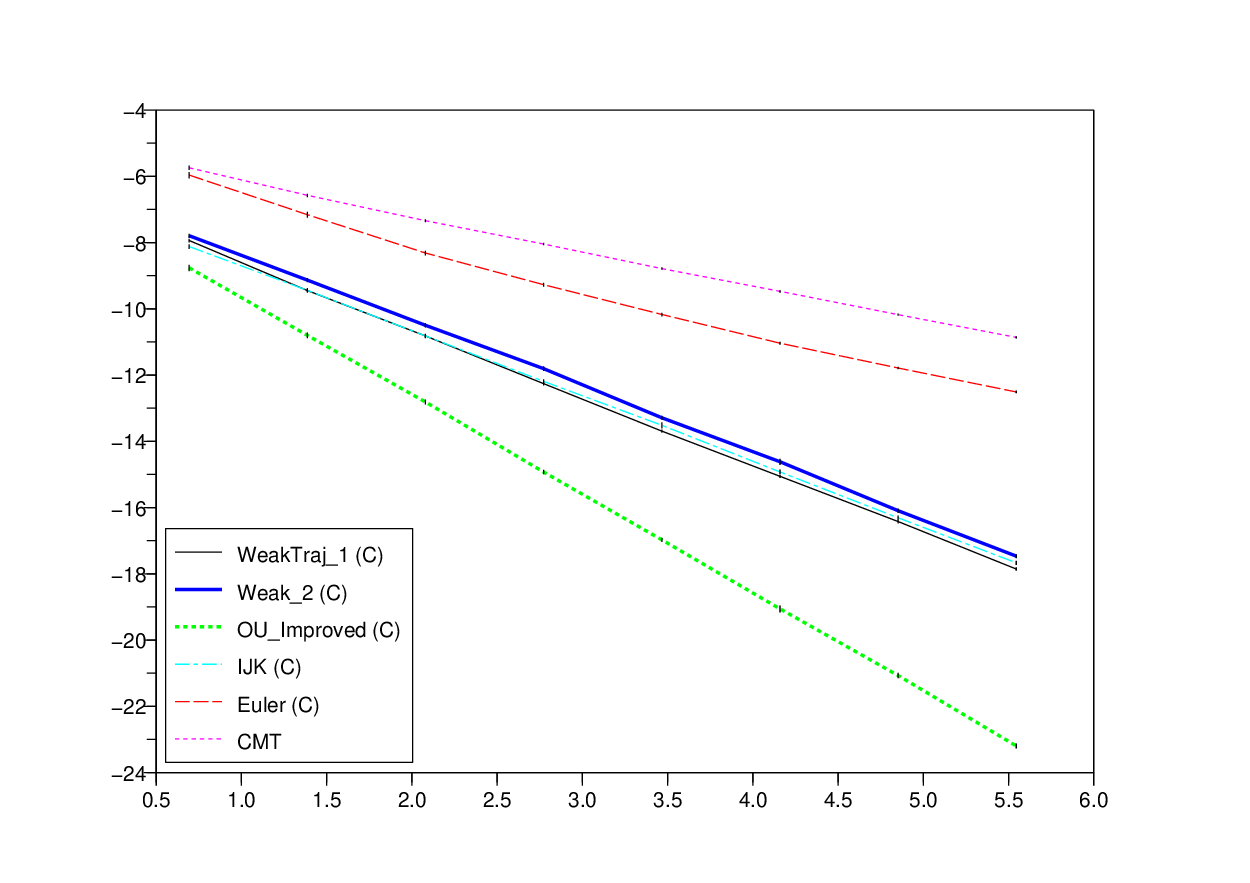}

\vspace{-6mm}
     \caption{Convergence at terminal time for the log-asset}
     \label{fig:strongTLogAsset}
  \end{minipage} \hfill
  \begin{minipage}[c]{.46\linewidth}
\hspace{-3mm}\includegraphics[scale=0.4]{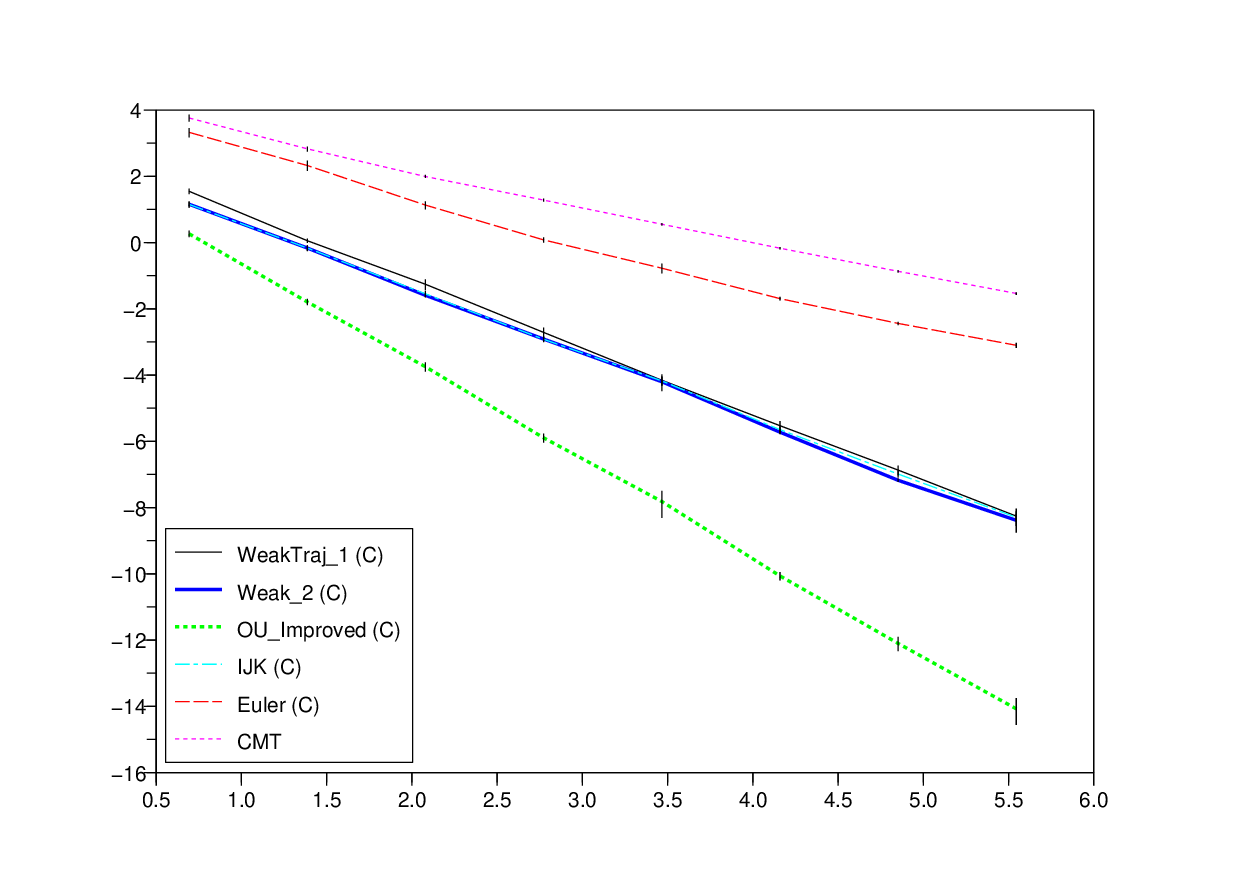}

\vspace{-6mm}
\caption{Convergence at terminal time for the asset}
       \label{fig:strongTAsset}
  \end{minipage}
\end{figure}

We observe that, as stated in Remark \ref{3demi}, the OU\_Improved
scheme exhibits a convergence rate of order $\frac{3}{2}$,
outperforming all the other schemes. As previously, the WeakTrak\_1
scheme exhibits a first order convergence rate. Note also that this
new coupling at terminal time improved the convergence rate of the
Weak\_2 and the IJK schemes up to order one and, surprisingly, it
improved the convergence rate of the Euler scheme up to an order
strictly greater than the expected $\frac{1}{2}$, approximately
$0.67$.

\subsection{Standard call pricing}
\subsubsection{Numerical illustration of weak convergence}
We compute the price of a call option with strike $K=100$ and
maturity $T=1$. For all the schemes but the CMT scheme, we use the
conditioning variance reduction technique presented in Remark
\ref{conditioning}.

In Figure \ref{fig:VitConvCall} we draw the logarithm of the pricing
error :
$\log\left(\left|P_{\text{exact}}-P^N_{\text{scheme}}\right|\right)$
where $P_{\text{exact}}\approx 12.82603$  is obtained by a
multilevel Monte Carlo with an accuracy of $5\times 10^{-4}$, as a function of
the logarithm of the number of time steps. In order to avoid
statistical noise, we make $10^7$ simulations.

We see that, as expected, the Weak\_2 scheme and the OU\_Improved
scheme exhibit a weak convergence of order two and converge much
faster than the others. The weak scheme already gives an accurate
price with only four time steps. The WeakTraj\_1 scheme has a weak
convergence of order one like the Euler and the IJK scheme, but it
has a greater leading error term. Fortunately, its better strong
convergence properties enable it to catch up with the multilevel
Monte Carlo method as we will see hereafter.

We also repeat this numerical experiment with the Stein\&Stein and
the quadratic Gaussian models (see figures \ref{fig:ConvCallSS} and
\ref{fig:ConvCallQG}) and check that the same conclusions hold.

Finally, note that the weak scheme does not require the simulation
of additional terms when compared to the Euler or the IJK schemes.
Combined with its second order weak convergence order, this makes
the Weak\_2 scheme very competitive for the pricing of plain vanilla
European options. In figure \ref{fig:ConvCall}, we give the relative
error of each scheme as a function of the computation time needed
when we fix the number of simulations to $M=100\,000$. We see that
both the Euler and the IJK scheme take five seconds to reach the
relative error obtained with the Weak\_2 scheme and the OU\_Improved
in less than a second. Note finally that the confidence interval is
much larger for the CMT scheme than for the other schemes because of
the use of the conditioning variance reduction technique for these
schemes.

\vspace{-2mm}

\begin{figure}
  \begin{minipage}[c]{.46\linewidth}
\hspace{-6mm}\includegraphics[scale=0.4]{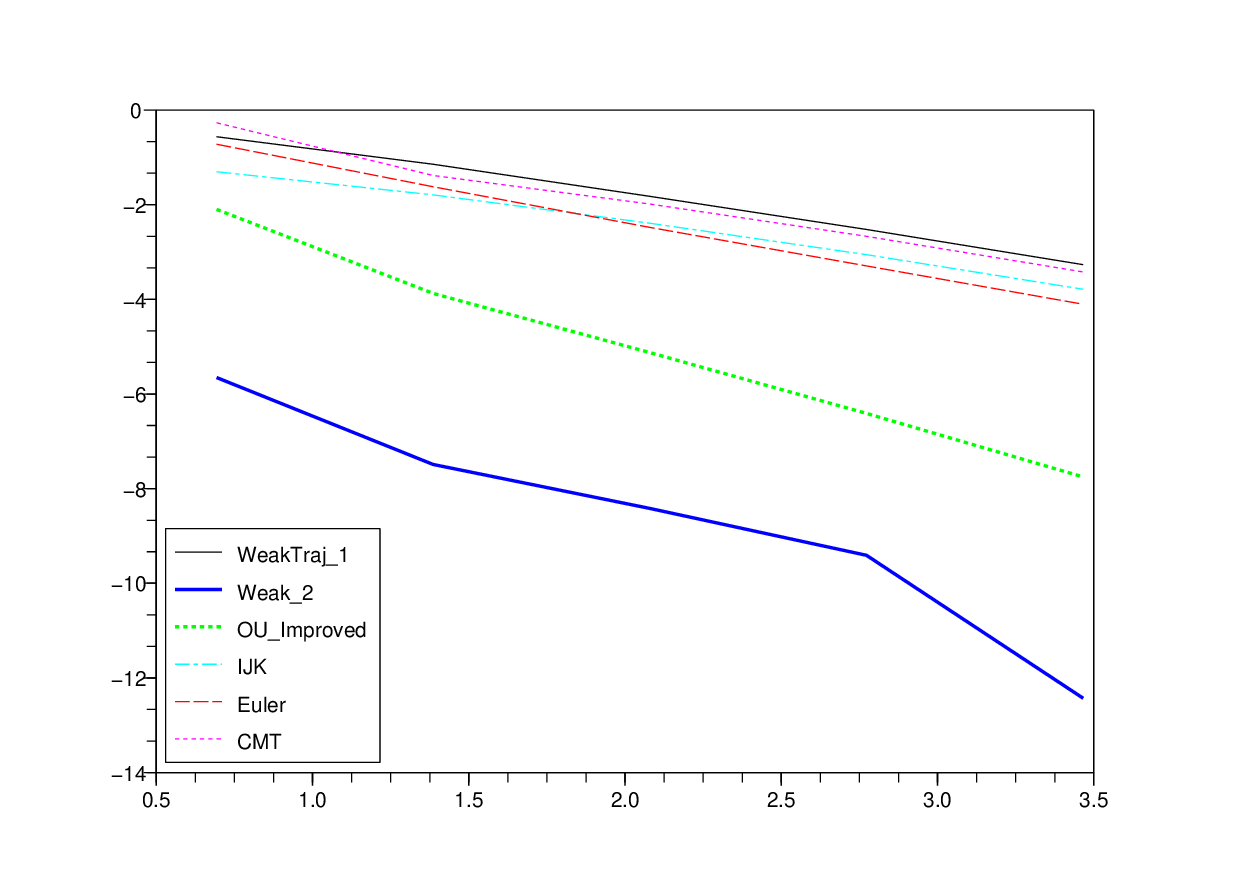}

\vspace{-6mm}
\caption{Illustration of the convergence rate for the call
option}
       \label{fig:VitConvCall}
  \end{minipage} \hfill
  \begin{minipage}[c]{.46\linewidth}
\hspace{-3mm}\includegraphics[scale=0.4]{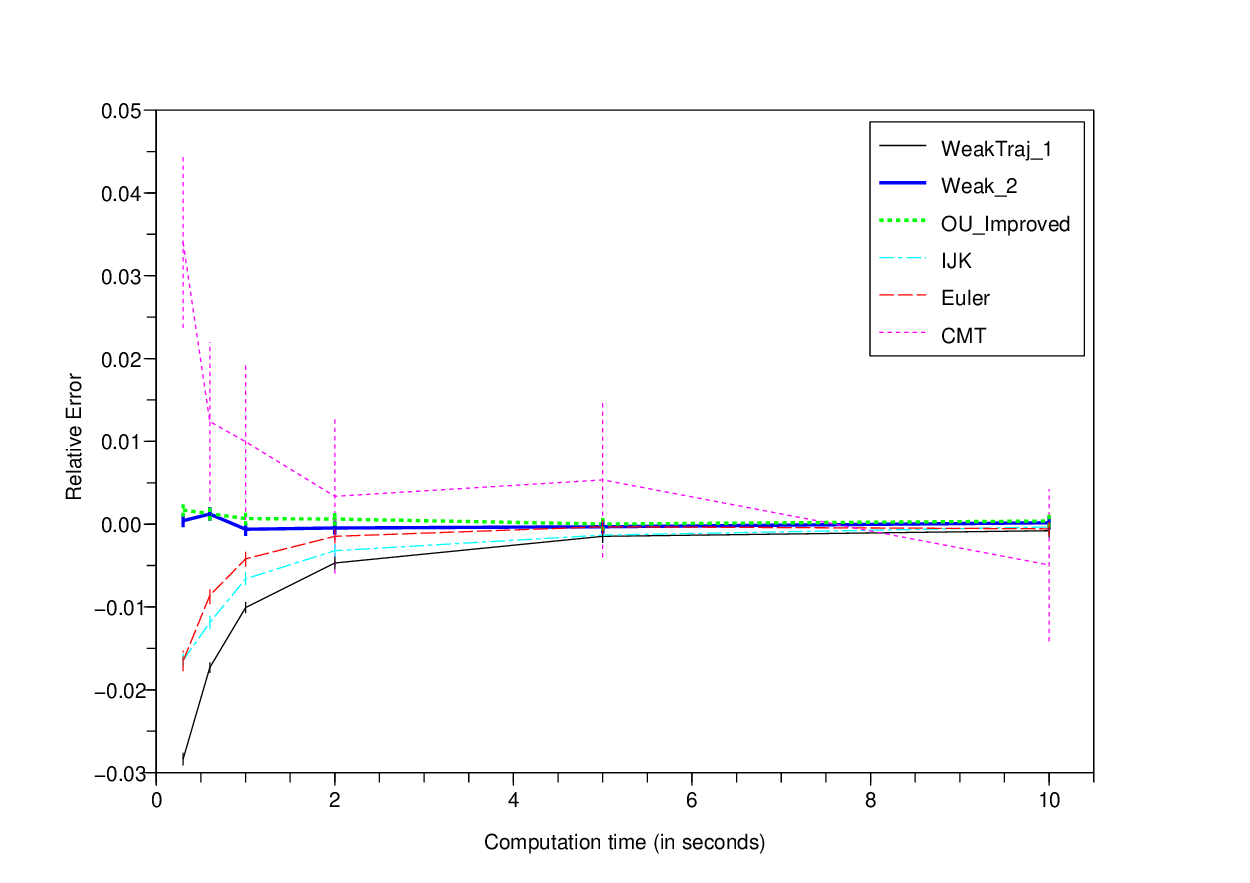}

\vspace{-6mm}
     \caption{Convergence of the
call price with respect to time}
     \label{fig:ConvCall}
  \end{minipage}
\end{figure}

\vspace{-2mm}

\begin{figure}
  \begin{minipage}[c]{.46\linewidth}
\hspace{-6mm}\includegraphics[scale=0.4]{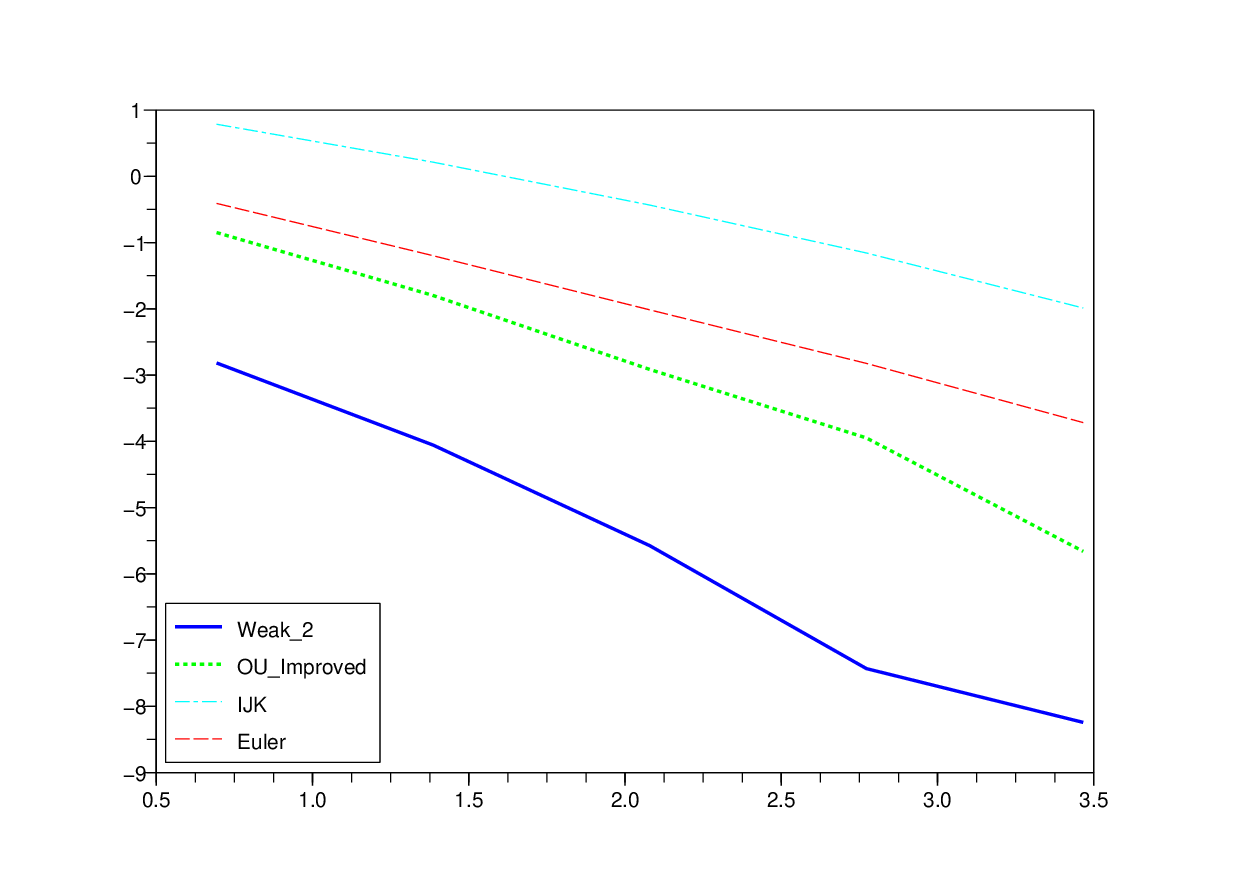}

\vspace{-6mm}
\caption{Stein\&Stein model - Illustration of the
convergence rate for the call option}
       \label{fig:ConvCallSS}
  \end{minipage} \hfill
  \begin{minipage}[c]{.46\linewidth}
\hspace{-3mm}\includegraphics[scale=0.4]{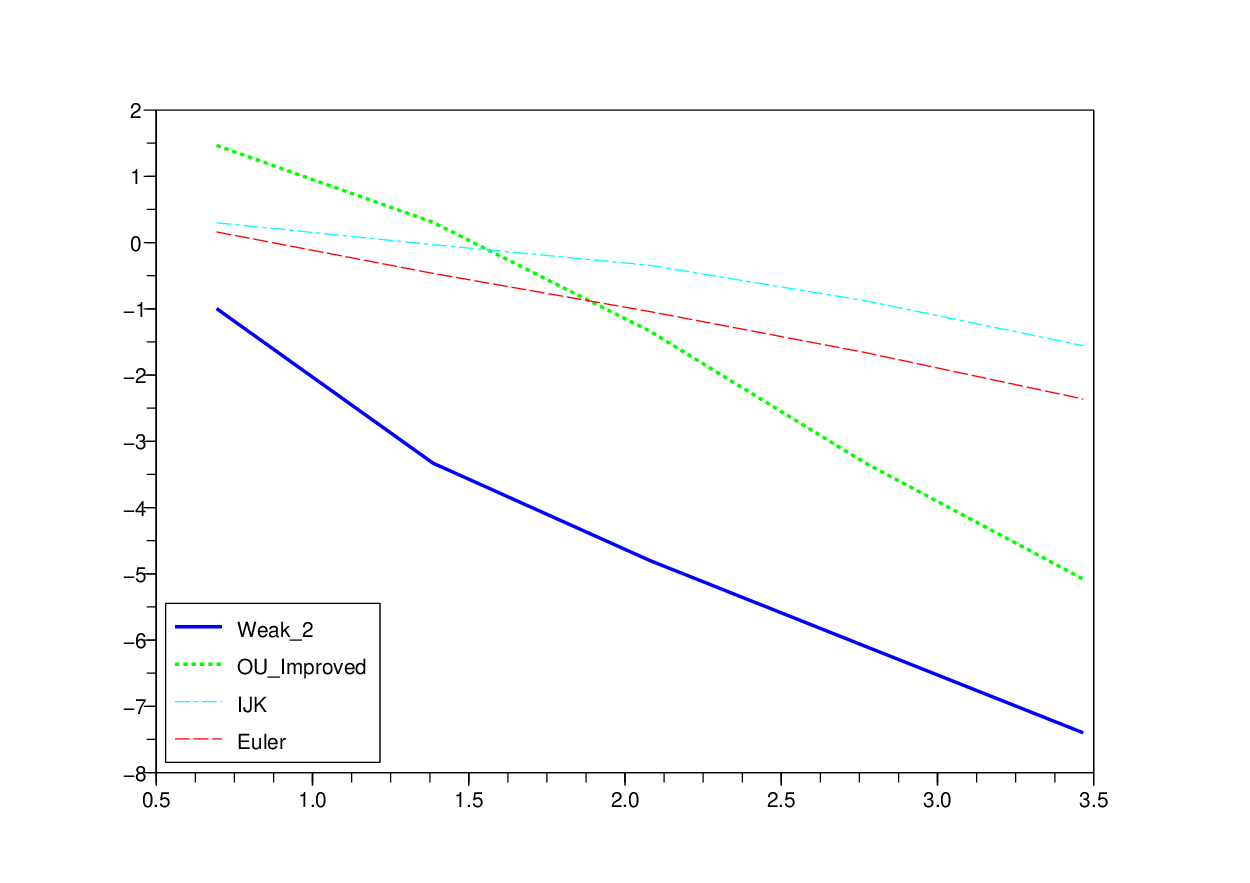}

\vspace{-6mm}
     \caption{Quadratic Gaussian model - Illustration of the convergence rate for the call option}
     \label{fig:ConvCallQG}
  \end{minipage}
\end{figure}

\subsubsection{Multilevel Monte Carlo}
Let us now apply the multilevel Monte Carlo method of Giles \cite{Giles1} to
compute the Call price. As previously, we consider the schemes straight at the
terminal time and use a conditioning variance reduction technique. We give the
CPU time as a function of the accuracy parameter Epsilon in Figure
\ref{fig:MultilevelCall}. This accuracy parameter is slightly higher than the
root mean square error achieved (see section 4.2 of \cite{Giles1} for details
on the heuristic numerical algorithm which is used). We check this numerically
by computing different ratios between the root mean square error achieved using
the reference value $P_{\text{exact}}$ and the target accuracy Epsilon (see
table \ref{tab:check}).

\begin{table}[!h]
\begin{center}
\begin{tabular}{|c|c|c|c|c|c|}
\hline &WeakTraj\_1&Weak\_2&OU\_Improved&IJK&Euler\\\hline
Epsilon=$10^{-1}$& 0.96 & 0.53& 0.6& 0.98& 0.61\\
\hline
Epsilon=$10^{-2}$& 0.8& 0.85&0.85&0.94&0.81\\
\hline
Epsilon=$10^{-3}$& 0.6&0.57&0.6&0.64 &0.7\\
\hline
Epsilon=$10^{-4}$& 0.8&0.98&0.54&0.6& 0.91\\
\hline
\end{tabular}
\end{center}
\vspace{-6mm} \caption{\label{tab:check} Ratio between rmse and Epsilon}
\end{table}

Figure \ref{fig:MultilevelCall} shows that both the Weak\_2 and the
OU\_Improved scheme are great time-savers. For the OU\_Improved scheme, the
effect coming from its good strong convergence properties is somewhat offset by
the additional terms that it requires to simulate. We can see nevertheless that it
is going  to overcome the Weak\_2 scheme for higher accuracy levels.

In order to illustrate the benefits of the multilevel Monte Carlo method, we
also give the variation of the computational complexity $C$, defined as the total number of timesteps performed on all levels (see section 5 of
\cite{Giles1}), with the desired accuracy with and without multilevel for the
OU\_Improved scheme (see figure \ref{fig:MultilevelCall_OU}).

\subsection{Lookback option pricing and multilevel Monte Carlo}
Finally, we consider an example of path-dependent option pricing : the lookback
option. More precisely, we compute the price of the option whose pay-off is
equal to $S_T-\min_{t\in[0,T]}S_t$. The use of multilevel Monte Carlo for
lookback options in local volatility models discretized by the Euler scheme was justified in \cite{GilesHighamMao}.

In order to take full advantage of the good convergence properties of our
schemes, we approximate the minimum of the scheme by the minimum of a drifted
Brownian motion. This is similar to what is done in \cite{Giles2}.

\newpage
\begin{figure}[!ht]
\begin{center}
\leavevmode \epsfig{file=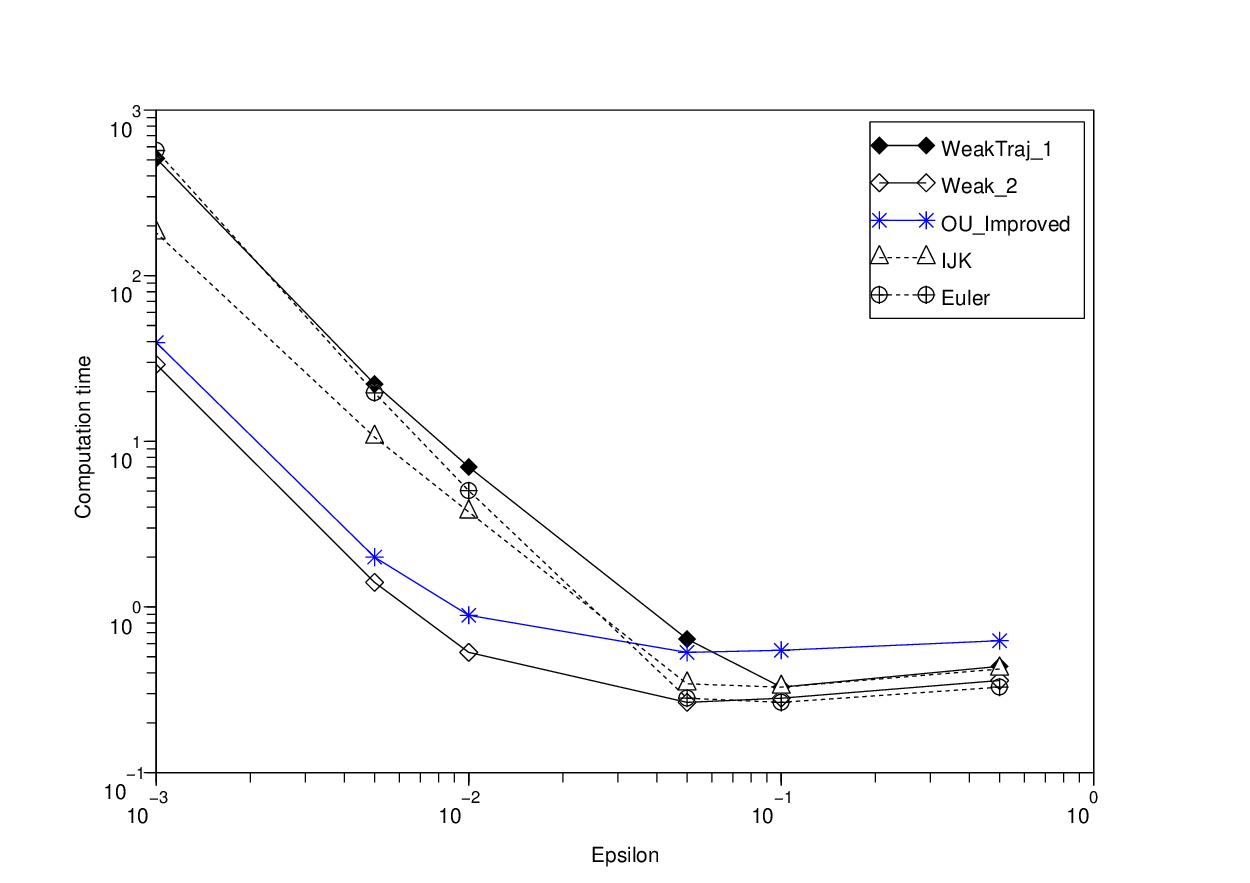,height=40ex}

\vspace{-6mm} \caption{Multilevel Monte Carlo method for a Call option using
different schemes} \label{fig:MultilevelCall}
\end{center}
\end{figure}

\begin{figure}[!ht]
\begin{center}
\leavevmode \epsfig{file=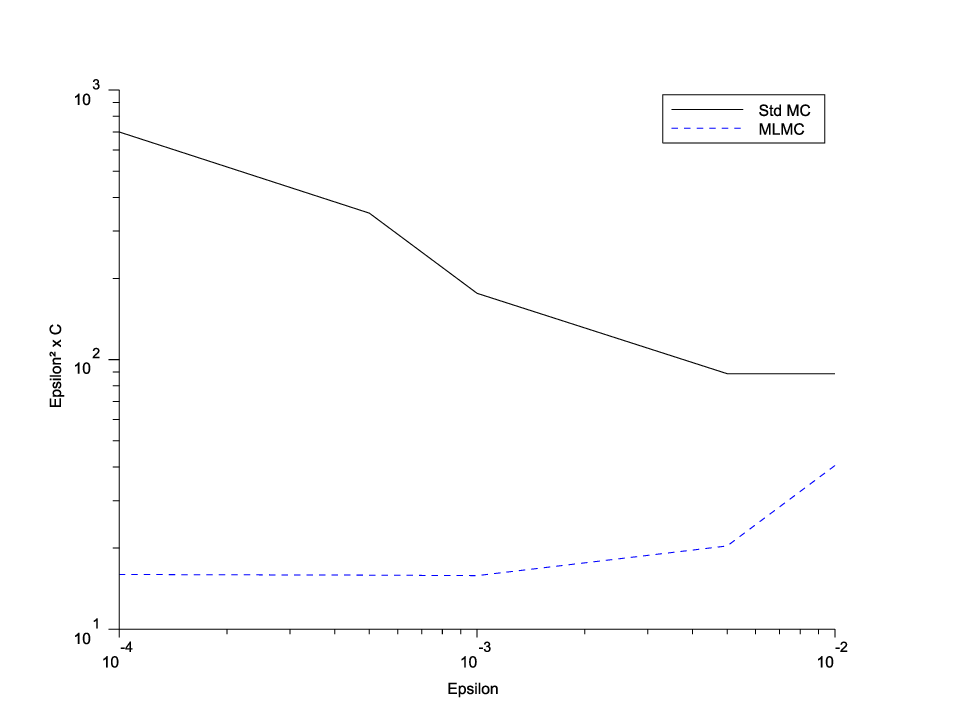,height=40ex}

\vspace{-6mm} \caption{OU\_Improved scheme with and without multilevel Monte
Carlo method} \label{fig:MultilevelCall_OU}
\end{center}
\end{figure}

More precisely, for the WeakTraj\_1 scheme, consider the interval
$[k\frac{T}{N},(k+1)\frac{T}{N}]$.\\[2mm]
\underline{Scheme with time step $\delta_{2N}$} :\\
We approximate $\min_{t\in[k\frac{T}{N},(k+1)\frac{T}{N}]}S_t$ by
$\widetilde{m}^{2N}_{2k}\wedge \widetilde{m}^{2N}_{2k+1}$ where,
$\forall 0\leq j\leq 2N-1$,
\[\widetilde{m}^{2N}_j=\frac{1}{2}\left(e^{\widetilde{X}^{2N}_{j\frac{T}{2N}}}+S^{e,2N}_{(j+1)\frac{T}{2N}}-\sqrt{\left(
e^{\widetilde{X}^{2N}_{j\frac{T}{2N}}}-S^{e,2N}_{(j+1)\frac{T}{2N}}\right)^2-2
e^{2\widetilde{X}^{2N}_{j\frac{T}{2N}}}f^2(Y_{j\frac{T}{2N}})\frac{T}{2N}\ln(U_{j})}\right),\]
where
$S^{e,2N}_{(j+1)\frac{T}{2N}}=e^{\widetilde{X}^{2N}_{j\frac{T}{2N}}}\left(1+r\frac{T}{2N}+f(Y_{j\frac{T}{2N}})\left(
\rho(W_{(j+1)\frac{T}{2N}}-W_{j\frac{T}{2N}})+\sqrt{1-\rho^2}(B_{(j+1)\frac{T}{2N}}-B_{j\frac{T}{2N}})\right)\right)$
and $(U_{j})_{0\leq j \leq 2N-1}$ is an independent sequence of
independent random variable uniformly distributed.\\[2mm]
\pagebreak
\begin{figure}[!ht]
\begin{center}
\leavevmode \epsfig{file=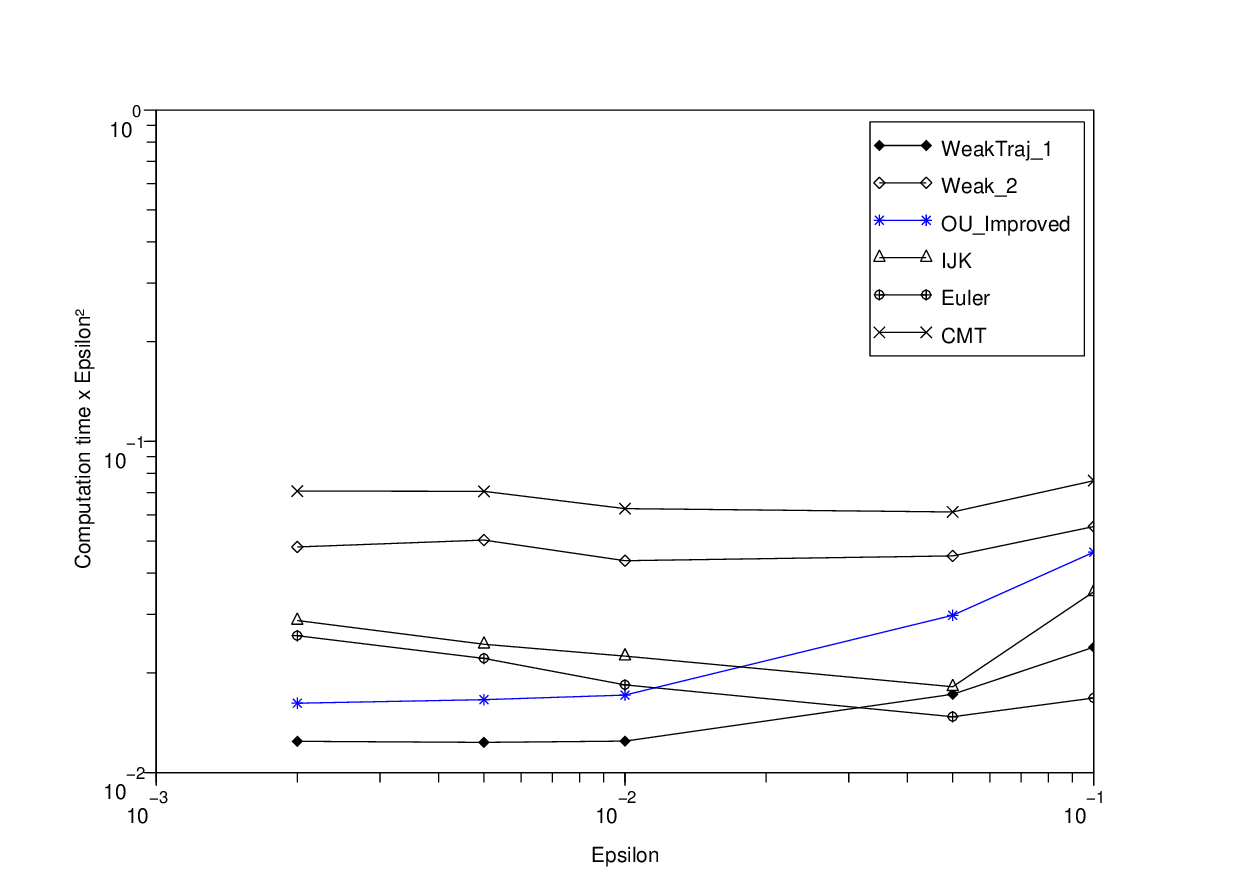,height=40ex}

\vspace{-6mm} \caption{Multilevel Monte Carlo method for a Lookback option
using different schemes.} \label{fig:MultilevelLook}
\end{center}
\end{figure}
$\,$\\
\underline{Scheme with time step $\delta_{N}$} :\\
According to Remark \ref{MLMC},
$\widetilde{\widetilde{X}}^N_{t_{k+1}}$ is computed using the
Brownian increment $\Delta \widetilde{B}^N_{k+1}$ given by a linear
combination of
$\left(B_{(2k+1)\frac{T}{2N}}-B_{k\frac{T}{N}},B_{(k+1)\frac{T}{N}}-B_{(2k+1)\frac{T}{2N}}\right)$
(see (\ref{couplage})). Now, to prevent bias, we are going to
approximate $\min_{t\in[k\frac{T}{N},(k+1)\frac{T}{N}]}S_t$ by the
minimum $\min_{[t_k,t_{k+1}]}\widetilde{S}^{e,2N}_t$ of some Euler
scheme $\widetilde{S}^{e,2N}_t$ like in the scheme with time step
$\delta_{2N}$. To remain consistent, we have to choose
\[\widetilde{S}^{e,2N}_{t_{k+1}}=e^{\widetilde{\widetilde{X}}^N_{t_k}}\left(1+r\frac{T}{N}+f(Y_{k\frac{T}{N}})\left(\rho(W_{t_{k+1}}-W_{t_k})+\sqrt{1-\rho^2}
\Delta \widetilde{B}^N_{k+1}\right)\right).\]

In order to ensure a good strong coupling with the scheme with time
step $\delta_{2N}$, we need to compute the intermediate value
$\widetilde{S}^{e,2N}_{(2k+1)\frac{T}{2N}}=e^{\widetilde{\widetilde{X}}^N_{t_k}}\left(1+r\frac{T}{N}+f(Y_{k\frac{T}{N}})\left(\rho(W_{t_{k+1}}-W_{t_k})+\sqrt{1-\rho^2}
\Delta \widetilde{\widetilde{B}}^{2N}_{2k+1}\right)\right)$ using
some Brownian increment $\Delta
\widetilde{\widetilde{B}}^{2N}_{2k+1}$ as close as possible to
$B_{(2k+1)\frac{T}{2N}}-B_{k\frac{T}{N}}$ but such that $\Delta
\widetilde{B}^N_{k+1}-\Delta \widetilde{\widetilde{B}}^{2N}_{2k+1}$
is independent of $\Delta \widetilde{\widetilde{B}}^{2N}_{2k+1}$ and
distributed according to $\mathcal{N}(0,\frac{T}{2N})$. Choosing
$\Delta \widetilde{\widetilde{B}}^{2N}_{2k+1}$ of the form
$\frac{a\left(B_{(2k+1)\frac{T}{2N}}-B_{k\frac{T}{N}}\right)+b\left(B_{(k+1)\frac{T}{N}}-B_{(2k+1)\frac{T}{2N}}\right)}{\sqrt{a^2+b^2}}$
and maximizing $Cov\left(\Delta
\widetilde{\widetilde{B}}^{2N}_{2k+1},B_{(2k+1)\frac{T}{2N}}-B_{k\frac{T}{N}}\right)=\frac{a}{\sqrt{a^2+b^2}}$
leads to $a=v^{2N}_{2k}+v^{2N}_{2k+1}$ and
$b=v^{2N}_{2k+1}-v^{2N}_{2k}$ (see Remark \ref{MLMC} for the
definition of $v^{2N}_{.}$).

Finally, we approximate
$\min_{t\in[k\frac{T}{N},(k+1)\frac{T}{N}]}S_t$ by
$\widetilde{m}^{N}_{k}\wedge \widetilde{m}^{N}_{k+1}$ where
\[\widetilde{m}^{N}_{k}=\frac{1}{2}\left(e^{\widetilde{\widetilde{X}}^{N}_{k\frac{T}{N}}}+\widetilde{S}^{e,2N}_{(2k+1)\frac{T}{2N}}-\sqrt{\left(
e^{\widetilde{\widetilde{X}}^{N}_{k\frac{T}{N}}}-\widetilde{S}^{e,2N}_{(2k+1)\frac{T}{2N}}\right)^2-2
e^{2\widetilde{\widetilde{X}}^{N}_{k\frac{T}{N}}}f^2(Y_{k\frac{T}{N}})\frac{T}{2N}\ln(U_{2k})}\right)\]
and
\[\widetilde{m}^{N}_{k+1}=\frac{1}{2}\left(\widetilde{S}^{e,2N}_{(2k+1)\frac{T}{2N}}+\widetilde{S}^{e,2N}_{(k+1)\frac{T}{N}}-\sqrt{\left(
\widetilde{S}^{e,2N}_{(2k+1)\frac{T}{2N}}-\widetilde{S}^{e,2N}_{(k+1)\frac{T}{N}}\right)^2-2
e^{2\widetilde{\widetilde{X}}^{N}_{k\frac{T}{N}}}f^2(Y_{k\frac{T}{N}})\frac{T}{2N}\ln(U_{2k+1})}\right).\]

The numerical results we obtain are very satisfactory. In figure
\ref{fig:MultilevelLook}, we draw the CPU time multiplied by the mean
square error against the root mean square error. We see that our schemes
perform much better than the others. \vspace{-2mm}

\section{Conclusion}
In this article, we have capitalized on the particular structure of
stochastic volatility models to propose and discuss two simple and
yet competitive discretization schemes. The first one exhibits first
order weak trajectorial convergence and has the advantage of
improving multilevel Monte Carlo methods for the pricing of path
dependent options. The second one is rather useful for pricing
European options since it has a second order weak convergence rate.

We have also focused on the special case of an Ornstein-Uhlenbeck process
driving the volatility, which encompasses many stochastic volatility models
such as the Scott's model \cite{Scott} or the quadratic Gaussian model. Then,
the convergence properties of the previous schemes are preserved when
simulating $(Y_t)_{0\leq t\leq T}$ exactly. We have also proposed an improved
scheme exhibiting both weak trajectorial convergence of order one and weak
convergence of order two.

Our numerical experiments confirm the theoretical rates of
convergence of our schemes. We also compare the time needed by the
different schemes to achieve a given precision in the multilevel
Monte Carlo computation of a plain vanilla Call option and a
lookback option. For high levels of precision our schemes turn out
to be more efficient than the Euler, the Kahl-Jäckel and the
Cruzeiro-Malliavin-Thalmaier schemes for both the vanilla Call and
the lookback option. The reason is that their better convergence
properties compensate the increase of computation effort at each
step.

As a last remark, we point out that our results can be naturally
extended to stochastic volatility models where the constant correlation
coefficient is replaced by a function $\rho(Y_t)$ of the process driving the stochastic
volatility in \eqref{SV1}. In this case, if one considers the transformation
$F(y)=\int_.^y \frac{\rho(z)f(z)}{\sigma(z)}dz$ and carries out the
same analysis then one should obtain weak trajectorial convergence
results under additional regularity assumptions on the function
$\rho$.

\bibliographystyle{plain}
\addcontentsline{toc}{section}{Bibliography}

\section{Appendix}
\subsection{Proof of Lemma \ref{MilsteinStrong}}

We first suppose that $p=1$. According to Theorem 5.2 page 72 of Milstein
\cite{Milstein}, it suffices to check that there exists a positive constant $C$
independent of $N$ such that
\begin{equation}\label{HypMil}
\begin{array}{rcl}
\ds
\left|\ee\left(Y_{\delta_N}-\overline{Y}^N_{\delta_N}\right)\right|&\leq&
\ds C \delta_N^2\\[3mm]
\ds
\left|\ee\left(\left(Y_{\delta_N}-\overline{Y}^N_{\delta_N}\right)^2\right)\right|^{\frac{1}{2}}&\leq&
\ds C \delta_N^{\frac{3}{2}}\\[3mm]
\ds
\left|\ee\left(\left(Y_{\delta_N}-\overline{Y}^N_{\delta_N}\right)^4\right)\right|^{\frac{1}{4}}&\leq&
\ds C \delta_N^{\frac{5}{4}}.
\end{array}
\end{equation}First note that
\[Y_{\delta_N}-\overline{Y}^N_{\delta_N}=\int_0^{\delta_N}
b(Y_s)-b(y_0)ds+\int_0^{\delta_N}\left(\int_0^s(b\sigma'+\frac{1}{2}\sigma^2
\sigma'')(Y_r)dr+(\sigma\sigma'(Y_r)-\sigma\sigma'(y_0))dW_r\right)dW_s.\]
Thanks to Itô's formula and to assumption ($\mathcal{H}$\ref{Mil1}),
we have that
\[\begin{array}{rcl}
\ds
\left|\ee\left(Y_{\delta_N}-\overline{Y}^N_{\delta_N}\right)\right|&=&\ds
\left|\int_0^{\delta_N}\int_0^s\ee\left((bb'+\frac{1}{2}b''\sigma^2)(Y_r)\right)dr
ds\right|\\[3mm]
&\leq&\ds C \left|\int_0^{\delta_N}\int_0^s C (1+\ee(|Y_r|^2))drds\right|\\[4mm]
&\leq&C \delta_N^2.
\end{array}\]
Using assumptions ($\mathcal{H}$\ref{Mil1}) and
($\mathcal{H}$\ref{Mil2}), we also have $\forall p\geq 1$
\[\begin{array}{rcl}
\ds
\ee\left(\left|Y_{\delta_N}-\overline{Y}^N_{\delta_N}\right|^{2p}\right)
&\leq&\ds 2^{2p-1}
\ee\left[\left|\int_0^{\delta_N}b(Y_s)-b(y_0)ds\right|^{2p}\right.\\[3mm]
&&\ds
\quad\quad\left.+\left|\int_0^{\delta_N}\left(\int_0^s(b\sigma'+\frac{1}{2}\sigma
\sigma'')(Y_r)dr+(\sigma\sigma'(Y_r)-\sigma\sigma'(y_0))dW_r\right)dW_s\right|^{2p}\right]\\[3mm]
&\leq&\ds 2^{2p-1}
\left[\delta_N^{2p-1} \int_0^{\delta_N}\ee\left(|b(Y_s)-b(y_0)|^{2p}\right)ds\right.\\[3mm]
&&\ds
\left.+C\delta_N^{p-1}\int_0^{\delta_N}\ee\left(\left|\int_0^s(b\sigma'+\frac{1}{2}\sigma
\sigma'')(Y_r)dr+(\sigma\sigma'(Y_r)-\sigma\sigma'(y_0))dW_r\right|^{2p}\right)ds\right]\\[3mm]
&\leq&\ds C \left[\delta_N^{2p-1}
\int_0^{\delta_N}s^pds+\delta_N^{p-1}\int_0^{\delta_N}
s^{2p-1}\int_0^s\ee\left(\left|(b\sigma'+\frac{1}{2}\sigma
\sigma'')(Y_r)\right|^{2p}\right)dr\,ds\right.\\[3mm]
&&\ds \quad\quad\left.+
\delta_N^{p-1}\int_0^{\delta_N}s^{p-1}\int_0^s
\ee\left(\left|\sigma\sigma'(Y_r)-\sigma\sigma'(y_0)\right|^{2p}\right)
dr\,ds\right]\\[3mm]
&\leq&C \delta_N^{3p}.\end{array}\] This implies both the second and the third
inequality of (\ref{HypMil}). This estimation is also sufficient to extend the
result of Milstein \cite{Milstein} to the $L^{2p}$ norm and conclude the proof.

\subsection{Proof of Lemma \ref{MomentExpo}}

One can easily check that $(Y_t)_{0\leq t\leq T}$ is a Gaussian
process which has the same distribution law as the process
$(y_0e^{-\kappa t}+\theta(1-e^{-\kappa t})+\frac{\nu e^{-\kappa
t}}{\sqrt{2\kappa}} W_{e^{2\kappa t}-1})_{0\leq t\leq T}$. So,

\[\begin{array}{rcl}
\ds \ee\left(e^{c_1 \sup_{0\leq t \leq T}
|Y_t|^{1+c_2}}\right)&=&\ds \ee\left(e^{c_1 \sup_{0\leq t \leq T}
|y_0e^{-\kappa t}+\theta(1-e^{-\kappa t})+\frac{\nu e^{-\kappa
t}}{\sqrt{2\kappa}} W_{e^{2\kappa
t}-1}|^{1+c_2}}\right)\\[3mm]
&\leq&\ds C \ee\left(e^{C\sup_{0\leq t\leq T}|W_{e^{2\kappa
t}-1}|^{1+c_2}}\right).
\end{array}\]
Since $ \sup_{0\leq t\leq e^{2\kappa T}-1 }|W_t| =\left(\sup_{0\leq
t\leq e^{2\kappa T}-1}W_t\right) \vee \left(-\inf_{0\leq t\leq
e^{2\kappa T}-1 }W_t\right)$, we deduce from the symmetry property
of the Brownian motion that
\[\begin{array}{rcl}
\ds \ee\left(e^{c_1 \sup_{0\leq t \leq T}
|Y_t|^{1+c_2}}\right)&\leq&\ds  C \ee\left(e^{C |\sup_{0\leq t\leq
e^{2\kappa T}-1}W_t|^{1+c_2}}+e^{C |\inf_{0\leq t\leq e^{2\kappa
T}-1 }W_t|^{1+c_2}}\right)\\[2mm]
&\leq &\ds 2C \ee\left(e^{C |\sup_{0\leq t \leq e^{2\kappa T}-1}
W_t|^{1+c_2}}\right). \end{array}\] The probability density function of
$\sup_{0\leq t \leq T} W_t$ is equal to $y\mapsto \sqrt{\frac{2}{\pi T}}
e^{-\frac{y^2}{2T}} \mathbb{1}_{\{y >0\}}$ (see for example problem 8.2 p. 96
of Karatzas and Shreve \cite{KaratzasShreve}) which permits to conclude.

Let us now assume that $y_0\neq 0$. Then $t_0\stackrel{\rm
def}{=}\inf\{t\in[0,T]:y_0e^{-\kappa t}+\theta(1-e^{-\kappa t})=0\}$
(convention $\inf\emptyset=T$) is positive. $\pp[|Y_0|\leq
N^{-\alpha}]=0$ for $N$ large enough and for $t\in (0,T]$,
\begin{align*}
\pp[|Y_t|\leq N^{-\alpha}]\leq \frac{2N^{-\alpha}}{\nu}
\sqrt{\frac{\kappa}{\pi(1-e^{-2\kappa t})}}\exp\left(-\frac{\kappa(|y_0e^{-\kappa t}+\theta(1-e^{-\kappa t})|-N^{-\alpha})^2}{\nu^2(1-e^{-2\kappa t})}\right).
\end{align*}
Since $\lim_{N\to\infty}\inf_{t\in[0,\frac{t_0}{2}]}(|y_0e^{-\kappa
t}+\theta(1-e^{-\kappa t})|-N^{-\alpha})>0$, one deduces that
$\sup_{t\in[0,\frac{t_0}{2}]}\pp[|Y_t|\leq N^{-\alpha}]\leq C
N^{-\alpha}$. The same conclusion holds for $\sup_{t\in
[\frac{t_0}{2},T]}\pp[|Y_t|\leq N^{-\alpha}]$ by bounding the
exponential factor by $1$.
\end{document}